\documentclass{article}
\usepackage[]{graphicx}
\usepackage[]{amsmath,amssymb}
\usepackage[]{amsthm,enumerate}
\usepackage{verbatim,bm}
\usepackage[dvipdfm]{hyperref}

\newtheorem{theorem}{Theorem}[section]
\newtheorem{lemma}[theorem]{Lemma}
\newtheorem{corollary}[theorem]{Corollary}
\theoremstyle{definition}
\newtheorem{definition}[theorem]{Definition}
\newtheorem{example}[theorem]{Example}
\newcommand{\defn}[1]{{\em #1}}
\theoremstyle{remark}
\newtheorem{remark}[theorem]{Remark}

\newcommand{\dd}{\; \mathrm{d}}
\newcommand\sbs{\subseteq}
\newcommand\ip[2]{\left\langle#1,#2\right\rangle}
\newcommand\one{{\bf1}}

\newcommand\tr{\opk{tr}}
\newcommand\abs[1]{\left|#1\right|}
\newcommand\cxi{\mathrm{i}}

\newcommand\cA{{\mathcal A}}
\newcommand\cE{{\mathcal E}}
\newcommand\cT{{\mathcal T}}
\newcommand\cU{{\mathcal U}}
\newcommand\cx{{\mathbb C}}
\newcommand\nn{{\mathbb N}}
\newcommand\zz{{\mathbb Z}}
\newcommand\re{{\mathbb R}}

\newcommand\al{\alpha}
\newcommand\be{\beta}
\newcommand\de{\delta}
\newcommand\De{\Delta}
\newcommand\om{\omega}
\newcommand\Om{\Omega}

\newcommand\opk[1]{\mathop{\mathrm{#1}}\nolimits}
\newcommand\Harm{\opk{Harm}}
\newcommand\Hom{\opk{Hom}}
\newcommand\Sym{\mathrm{Sym}}
\renewcommand{\Re}{\opk{Re}}
\renewcommand{\Im}{\opk{Im}}
\newcommand{\cl}{\opk{cl}}
\newcommand{\spn}{\opk{span}}

\newcommand\la{\lambda}
\newcommand\co[1]{\overline{#1}}
\newcommand\bal{\bm{\alpha}}
\newcommand\bbe{\bm{\beta}}
\newcommand\bla{\bm{\lambda}}

\newcommand\cS{{\mathcal S}}
\newcommand\cV{{\mathcal V}}

\begin{document}

\title{Complex spherical designs and codes}
\date{}

\author{Aidan Roy \\ Department of Combinatorics and Optimization \& \\ Institute for Quantum Computing, \\ University of Waterloo \\ aproy@uwaterloo.ca \and
Sho Suda \\ Graduate School of Information Sciences, \\ Tohoku University \\ suda@ims.is.tohoku.ac.jp}

\maketitle

Abstract: Real spherical designs and real and complex projective designs have been shown by Delsarte, Goethals, and Seidel to give rise to association schemes when the strength of the design is high compared to its degree as a code. In contrast, designs on the complex unit sphere remain relatively uninvestigated, despite their importance in numerous applications. In this paper we develop the notion of a complex spherical design and show how many such designs carry the structure of an association scheme. In contrast with the real spherical designs and the real and complex projective designs, these association schemes are nonsymmetric.

\tableofcontents

\section{Introduction} 

In the 1970's, Delsarte, Goethals and Seidel~\cite{dgs,dgs2} introduced the notions of real spherical designs and real and complex projective designs as a generalization of finite incomplete block designs. Roughly speaking, a spherical design of strength $t$ is a finite set of points on a unit sphere such that any polynomial of degree $t$ has the same average value on those points as it does on the entire sphere. In their seminal paper, Delsarte et al.~found a number of lower bounds on the size of a design, and showed that designs which are tight to those bounds carry the structure of an association scheme. They also introduced the notions of a spherical or projective code, which is a finite collection of points on the unit sphere or in projective space such that only small number of distinct angles occur between points. They found upper bounds on the size of a code, and pointed out a close relationship between codes and designs. 

Real spherical designs and real and complex projective designs have been studied often in the ensuing years, both for their interesting combinatorics properties and because of a variety of applications in information theory. However, the notion of a complex spherical design, in which one chooses points on a complex unit sphere rather than in complex projective space, remains relatively uninvestigated. In this paper, we fill some of this void. Given a set of unit vectors $X$ in $\cx^d$, we work with its inner product set
\[
A(X) := \{a^*b : a, b \in X, a \neq b \}
\]
rather than the absolute values of the inner products, as one would in the projective case. We find bounds similar to those of Delsarte et al.~and show that in many cases, association schemes again arise.

There are a number of reasons why complex spherical designs have, until now, been left out of the picture. Firstly, it is not at all obvious that complex spherical designs are a non-trivial extension of complex projective designs. In other words, are there complex spherical designs whose corresponding points in projective space are not projective designs of an analogous strength? We show that indeed spherical designs can have a richer structure, and further that the association schemes they carry can be nontrivial. The second reason for the neglect is that the technical machinery required to investigate designs, that of harmonics and zonal polynomials, does not work as nicely as in the projective or real spherical cases. In particular, the complex unit sphere with this choice of inner product is neither a Delsarte space \cite{neumaier} nor a compact metric space \cite{lev}, conditions which allow the theory of Delsarte et al.~to work quite generally. Nevertheless, several authors have extended the Delsarte techniques to other spaces~\cite{bbc,roy,rs}, and here we show how they can be applied to the complex unit sphere as well.

Collections of points on the complex unit sphere have a number of important applications. One of the reasons they have been studied intensively in the last few years is their use as measurements in quantum information theory~\cite{sco,rs2}. Roughly speaking, any complex projective $1$-design may be considered a quantum measurement, and certain types of projective $2$-designs such as ``mutually unbiased bases" and ``symmetric informationally complete POVMs", are optimal measurements for use in quantum state tomography~\cite{rbsc}. However, good complex projective designs play a natural role in a variety of other combinatorial applications such as sphere-packing or complex root systems~\cite[Chapter 3]{cs}. For these problems, complex spherical designs are also a natural structure to consider.

The main results of the paper are:
\begin{itemize}
\item Theorem~\ref{thm:relbound}, which establishes Delsarte linear programming bounds for complex spherical $\cT$-designs;
\item Theorem~\ref{thm:S1}, which shows that certain complex spherical $\cT$-designs with only a small number of inner product values carry association schemes; 
\item Theorem~\ref{thm:S8}(2), which gives sufficient conditions for antipodal $\cT$-designs (designs that are invariant under pointwise multiplication by an $n$-th root of unity) to carry an association scheme; and 
\item Theorem~\ref{designAS}, which shows that some parameters of a nonsymmetric association scheme can be characterized using complex designs constructed by treating the idempotents of the scheme as multiples of Gram matrices.
\end{itemize}

The outline of the paper is as follows. Section \ref{sec:polynomials} reviews the technical machinery of harmonic polynomials required to discuss the complex unit sphere. 
In Section \ref{sec:designs}, we introduce the notion of a complex spherical design and how it is related to other types of designs, such as real spherical designs and complex projective designs. 
In Sections \ref{sec:complexcode} and \ref{sec:bounds} we define code on complex unit sphere as a set of vectors with few inner product values, and prove upper bounds on the size of a design and lower bounds on the size of a code. We also give conditions for the bounds to be tight.
Section \ref{sec:schemes} introduces nonsymmetric association schemes and shows that for certain designs, the relations defined by the inner products carry the structure of a scheme.
Sections \ref{sec:connections} and \ref{sec:antipodal} explain how the schemes arising from complex spherical designs relate to the schemes from real or projective designs. 
In Section \ref{sec:designsfromschemes} we show that certain nonsymmetric association schemes naturally give rise to small strength complex spherical designs, similar to the way symmetric association schemes produce real spherical designs. 
The remaining sections discuss designs of particular type: derived designs, which are constructed by deleting points; designs constructed from orbits of finite subgroups in unitary group; and sporadic examples of association schemes arising from designs which do not fit into any earlier results.

\section{Polynomials of the unit sphere}\label{sec:polynomials}
For further details on harmonic polynomials, see~\cite{dgs,rud}.

Let $\Om(d)$ denote the unit sphere in $\cx^d$. Then for $z = (z_1,\ldots,z_d)$ in $\Om(d)$, we let $\Hom(k,l)$ denote the set of polynomials 
in $z$ which are homogeneous of degree $k$ in $\{z_1,\ldots,z_d\}$ and homogeneous of degree $l$ in $\{\overline{z_1},\ldots,\overline{z_d}\}$, where $\overline{z_i}$ is the complex conjugate of $z_i$. Since $\sum_{i=1}^d z_i\overline{z_i} = z^*z = 1$, it follows that $\Hom(k,l)$ is contained in $\Hom(k+1,l+1)$. 

The space $\Hom(k,l)$ is a representation of $U(d)$, the group of $d \times d$ unitary matrices, with the following action: for $U \in U(d)$, $f \in \Hom(k,l)$, and $z \in \Om(d)$, $(Uf)(z) := f(U^*z)$. This representation decomposes into irreducible representations called \defn{harmonic spaces}. $\Harm(k,l)$ is defined as the subspace of $\Hom(k,l)$ in the kernel of the Laplacian operator $\De = \sum_{i=1}^d \partial^2/\partial z_i \partial \overline{z_i}$, and $\Hom(k,l)$ decomposes in the following way:
\[
\Hom(k,l) = \bigoplus_{i=0}^{\min\{k,l\}} \Harm(k-i,l-i).
\]
Since
\[
\dim(\Hom(k,l)) = {d+k-1 \choose d-1}{d+l-1 \choose d-1},
\]
the decomposition implies that
\[
\dim(\Harm(k,l)) = {d+k-1 \choose d-1}{d+l-1 \choose d-1} - {d+k-2 \choose d-1}{d+l-2 \choose d-1}.
\]
We define $m_{k,l} := \dim(\Harm(k,l))$.

Define an inner product on functions $f$ and $g$ on $\Om(d)$ as follows:
\[
\ip{f}{g} := \int_{\Om(d)} \overline{f(z)}g(z) \dd z.
\]
Here $\mathrm{d}z$ is the unique invariant Haar measure on $\Om(d)$, normalized so that $\int_{\Om(d)}\mathrm{d}z = 1$. With respect to this inner product, $\Harm(k,l)$ is orthogonal to $\Harm(k',l')$ whenever $(k,l) \neq (k',l')$.

Let $\nn$ denote the set of nonnegative integers. For each $(k,l) \in \nn^2$, we define a Jacobi polynomial $g_{k,l}$ as follows:
\[
g_{k,l}(x) := \frac{d+k+l-1}{(d-1)!} \sum_{r=0}^{\min\{k,l\}} (-1)^r \frac{(d+k+l-r-2)!}{r!(k-r)!(l-r)!}x^{k-r}\overline{x}^{l-r}.
\]
These polynomials are orthogonal in the following sense: for any fixed $a \in \Om(d)$, the function 
\[
g_{k,l,a}: z \mapsto g_{k,l}(a^*z)
\] 
is in $\Harm(k,l)$. Moreover, any two such maps $g_{k,l,a}$ and $g_{k',l',a'}$ are orthogonal for $(k,l) \neq (k',l')$, since $\Harm(k,l)$ and $\Harm(k',l')$ are orthogonal. The maps $g_{k,l,a}$ are sometimes called the zonal orthogonal polynomials for $\Harm(k,l)$.

Explicitly, the first few Jacobi polynomials are 
\begin{align*}
g_{0,0}(x) & = 1, \\
g_{1,1}(x) & = (d+1)(dx\overline{x} - 1), \\
g_{2,2}(x) & = \frac{d(d+3)}{4}((d+1)(d+2)x^2\overline{x}^2 - 4(d+1)x\overline{x} + 2), \\
g_{1,0}(x) & = dx, \\
g_{2,1}(x) & = \frac{d(d+2)}{2}((d+1)x^2\overline{x} - 2x), \\
g_{3,2}(x) & = \frac{d(d+1)(d+4)}{12}((d+2)(d+3)x^3\overline{x}^2 - 6(d+2)x^2\overline{x} + 6x), \\
g_{2,0}(x) & = \frac{d(d+1)}{2}x^2.
\end{align*}
The Jacobi polynomials are also normalized so that $g_{k,l}(1) = \dim(\Harm(k,l))$. The polynomials $g_{k,k}$ are the Jacobi polynomials of the complex projective space discussed in Delsarte et al.~\cite{dgs}. Recursively, the Jacobi polynomials satisfy
\[
x g_{k,l}(x) = a_{k,l}g_{k+1,l}(x)+b_{k,l}g_{k-1,l}(x)
\]
where $a_{k,l}=\frac{k+1}{d+k+l}$ and $b_{k,l}=\frac{d+l-2}{d+k+l-2}$. 

The essential property of the Jacobi polynomials is the following theorem, known as Koornwinder's addition theorem~\cite{koo}. 

\begin{theorem}\label{thm:addition}
Let $\{e_1,\ldots,e_{m_{k,l}}\}$ be an orthonormal basis for the space $\Harm(k,l)$. Then for any $a,b \in \Om(d)$,
\[
\sum_{i=1}^{m_{k,l}} \overline{e_i(a)}e_i(b) = g_{k,l}(a^*b).
\]
\end{theorem}

From the addition theorem it follows that $g_{k,l,a} = \sum_{i=1}^{m_{k,l}} \overline{e_i(a)}e_i$, and, taking inner products, $\ip{g_{k,l,a}}{e_i} = e_i(a)$. Therefore for any $p \in \Harm(k,l)$,
\[
\ip{g_{k,l,a}(x)}{p(x)} = p(a).
\]
In other words, $g_{k,l,a}$ is the unique element of $\Harm(k,l)$ whose dual maps each polynomial in $\Harm(k,l)$ to its value at $a$. It follows immediately that $\{x \mapsto g_{k,l}(a^*x) : a \in \Om(d)\}$ spans $\Harm(k,l)$. It also implies that the Jacobi polynomials have a certain positive semidefinite property which is the basis for Delsarte's bounds. For any subset $X \sbs \Om(d)$,
\begin{equation}\label{eqn:zonalsum}
\sum_{a,b \in X} g_{k,l}(a^*b) = \sum_{a,b \in X} \ip{g_{k,l,b}}{g_{k,l,a}} = \ip{\sum_{b \in X} g_{k,l,b}}{\sum_{a \in X} g_{k,l,a}} \geq 0.
\end{equation}

Also note that that $g_{k,l}(x)$ can expressed in terms of the Hypergeometric function ${}_2F_1(a,b;c;x)$ or the usual Jacobi polynomial $P_n^{(\al,\be)}(x)$:
\begin{align*}
g_{k,l}(x)&=m_{k,l}x^k\co{x}^l{}_2F_1(-k,-l;d-1;1-\frac{1}{x\co{x}})\\
&=m_{k,l}\frac{l!(d-2)!}{(l+d-2)!}x^{k-l}P_l^{d-2,k-l}(2x\co{x}-1)\\
&=m_{k,l}\frac{k!(d-2)!}{(k+d-2)!}\co{x}^{l-k}P_k^{d-2,l-k}(2x\co{x}-1).
\end{align*}
It follows that 
\[
g_{k,l}(x)g_{k',l'}(x)=\sum_{i=0}^{\min\{k+k',l+l'\}}q_ig_{k+k'-i,l+l'-i}(x),
\]
for some positive constants $q_i$ such that if $k+k' = l+l'$, then $q_{k+k'} = m_{k,l}\de_{k,l'}$. (In other words, the coefficients in the Jacobi polynomial expansion of $g_{k,l}(x)g_{k',l'}(x)$ are nonnegative, and the coefficient of $g_{0,0}$ is nonzero if and only if $k = l'$ and $k'=l$.)

Given polynomials $F(x)$ and $H(x)$ in $\re[x,\overline{x}]$, expand each as a sum of Jacobi polynomials as follows:
$$F(x)=\sum_{(k,l)\in\nn^2}f_{k,l}g_{k,l}(x),\quad H(x)=\sum_{(k,l)\in\nn^2}h_{k,l}g_{k,l}(x).$$
From the facts about the coefficients $q_i$ above, we get the following.
\begin{lemma}\label{coef}
Let $H(x)=F(x)g_{m,n}(x)/g_{m,n}(1)$ for some $(m,n)\in\nn^2$.
\begin{itemize}
\item[(i)] If $f_{k,l}\geq0$ for all $(k,l)\in\nn^2$, then $h_{k,l}\geq0$ for all $(k,l)\in\nn^2$.
\item[(ii)] If $f_{k,l}>0$ for all $(k,l)\in\nn^2$, then $h_{k,l}>0$ for all $(k,l)\in\nn^2$.
\item[(iii)] $h_{0,0}=f_{n,m}$.
\end{itemize}
\end{lemma}

Finally, we require the average values of $g_{k,l}$ over the unit sphere. If $(k,l) \neq (0,0)$, then $\Harm(k,l)$ is orthogonal to the space of constant functions $\Harm(0,0)$. It follows that for any $a \in \Om(d)$,
\[
\int_{\Om(d)} g_{k,l}(a^*z) \dd z = \int_{\Om(d)} g_{k,l,a}(z) \dd z = \ip{1}{g_{k,l,a}} = 0.
\]

However, it is also possible to evaluate polynomials in $\Hom(k,l)$ directly. To do this it is useful to consider collections of points on the unit sphere that differ only by a complex root of unity.

\begin{definition} A set of points $X \sbs \Om(d)$ is \defn{$n$-antipodal} if for some $n$-th primitive root of unity $\om$ and some $L \sbs X$,
\[
(L,\om L, \ldots, \om^{n-1}L) \mbox{ is a partition of } X.
\]
We call $X$ an \defn{$n$-antipodal cover} of $L$.
\end{definition}

The following lemma shows that polynomials in $\Hom(k,l)$ often vanish when averaged over a set of $n$-antipodal points, provided that $k \neq l$.

\begin{lemma}\label{lem:avgantipodal}
Let $X$ be an $n$-antipodal set in $\Om(d)$ and let $k$ and $l$ be nonnegative integers such that $k \not\equiv l \mod n$. Then for any $f \in \Hom(k,l),$
\[
\sum_{x \in X} f(x) = 0.
\]
\end{lemma}

\begin{proof}
Let $\om$ be a primitive $n$-th root of unity, so $X = \cup_{i=0}^{n-1} \om^i L$ for some subset $L$. Since $f$ is homogeneous, $f(\om x) = \om^{k-l} f(x)$. Therefore,
\begin{align*}
\sum_{x \in X} f(x) & = \sum_{x \in L} \sum_{i=0}^{n-1} f(\om^i x) \\
& = \sum_{x \in L} f(x) \sum_{i=0}^{n-1} \om^{i(k-l)} \\
& = 0,
\end{align*}
since $\om^{k-l} \neq 1$ is another $n$-th root of unity.
\end{proof}

By averaging $n$-antipodal points in $\Om(d)$ for large enough $n$, a similar argument shows that if $f(z)$ is any monomial in $\Hom(k,l)$ with $k \neq l$, then
\[
\int_{\Om(d)} f(z) \dd z = 0.
\]
In fact, the average of a monomial $f \in \Hom(k,k)$ is also trivial unless the power of each coordinate $z_i$ is the same as the power of its conjugate $\overline{z_i}$. If on the other hand each $z_i$ and $\overline{z_i}$ have the same exponent, then a theorem from Rudin~\cite{rud} explains how to evaluate monomials in $\Hom(k,k)$. 

\begin{theorem}
If $f(z) = (z_1\overline{z_1})^{a_1}\ldots(z_d\overline{z_d})^{a_d}$ is a monomial in $\Hom(k,k)$ (so $\sum_{i=1}^d a_i = k$), then
\[
\int_{\Om(d)} f(z) \dd z = \frac{(d-1)!a_1!\ldots a_d!}{(d+k-1)!}.
\]
\end{theorem}

In particular, we can evaluate $\int_{\Om(d)} \abs{a^*z}^{2k} \dd z$ for any $a \in \Om(d)$ by noting that the integral is invariant of the choice of $a$. Letting $e_1$ denote the first standard basis vector in $\Om(d)$, 
\begin{eqnarray*}
\int_{\Om(d)} \abs{a^*z}^{2k} \dd z & = &\int_{\Om(d)} \abs{e_1^*z}^k \dd z = \int_{\Om(d)} (z_1\overline{z_1})^k \dd z \\
& = & \frac{(d-1)!k!}{(d+k-1)!} = {d+k-1 \choose d-1}^{-1}.
\end{eqnarray*}

\section{Complex spherical designs}\label{sec:designs}

We now introduce complex spherical designs. Like other designs, a finite set of points in $\Om(d)$ is a design if, for a certain class of polynomials, the average over $X$ is the same as the average over $\Om(d)$. 

In order to define complex spherical designs, we require a partial order on pairs of nonnegative integers $(k,l) \in \nn^2$, where $k$ and $l$ are the degrees of the bivariate polynomials in $x$ and $\overline{x}$. We define the \defn{product order} as follows: $(k,l) \preceq (m,n)$ if and only if $k \leq m$ and $l \leq n$. 

 A \defn{lower set} with respect to $\succeq$ is a finite set $\cT$ in $\nn^2$ such that if $(k,l)$ is in $\cT$, so is $(m,n)$ for all $(m,n) \preceq (k,l)$. Often we will specify a lower set as the closure of its maximal elements. For example, $\cT = \cl(\{(0,2),(1,1),(2,0)\})$ refers to the lower set $\{(0,0),(0,1),(0,2),(1,0),(1,1),(2,0)\}$.

\begin{definition}
Let $X$ be a finite subset of $\Om(d)$ and let $\cT$ be a lower set in $\nn^2$. Then $X$ is a \defn{complex spherical $\cT$-design} if, for every polynomial $f \in \Hom(k,l)$ such that $(k,l)$ is in $\cT$, 
\begin{equation}\label{eqn:klregular}
\frac{1}{|X|} \sum_{z \in X} f(z) = \int_{\Om(d)} f(z) \dd z.
\end{equation}
\end{definition}

We call a set $X$ \defn{$(k,l)$-regular}~\cite{mop} if every $f \in \Hom(k,l)$ satisfies $\eqref{eqn:klregular}$.
It is obvious that if $X$ is $(k,l)$-regular, then $X$ is $(l,k)$-regular. 

Delsarte $\cT$-design in association schemes whose indices of primitive idempotents has a structure of a poset was studied by \cite{mar}.
If we regard designs in real unit sphere as an analogue of designs in $Q$-polynomial schemes, designs in complex unit sphere may be regarded as an analogue of designs in such association schemes. 

Since $\cT$ is a lower set, it follows that if $X$ is a $\cT$-design, then $X$ is also a $\cT'$-design for every lower set $\cT'$ that is a subset of $\cT$. The following lemma, which is standard from design theory and an easy generalization of Delsarte et al.~\cite{dgs}, indicates that if suffices to check a small number of polynomial values to decide if a subset is a design. Let $H_{k,l}$ denote the \defn{$(k,l)$-characteristic matrix} of $X$: the rows and columns of $H_{k,l}$ are indexed by $X$ and an orthonormal basis $\{e_1,\ldots,e_{m_{k,l}}\}$ of $\Harm(k,l)$, with the $(x,i)$ entry of $H_{k,l}$ equal to $e_i(x)$.

\begin{lemma} \label{lem:designchar}
Let $X$ be a finite nonempty subset of $\Om(d)$ and let $\cT$ be a lower set in $\nn^2$. Then the following are equivalent:
\begin{itemize}
\item[(i)] $X$ is a complex spherical $\cT$-design.
\item[(ii)] $H^*_{k,l}H_{k',l'} = |X|\de_{k,k'}\de_{l,l'}I$ for all $k,l,k',l'$ such that $(k+l',l+k') \in \cT$.
\item[(iii)] $\sum_{x,y \in X} g_{k,l}(x^*y) = 0$ for all $(k,l) \neq (0,0)$ in $\cT$.
\end{itemize}
\end{lemma}

\begin{proof} Recall that $g_{k,l,x}: y \mapsto g_{k,l}(x^*y)$ is in $\Harm(k,l)$ and is therefore orthogonal to the constant function $1 \in \Harm(0,0)$. Therefore, $X$ is a $\cT$-design if and only if for all $(k,l) \neq (0,0)$ in $\cT$, 
\[
\sum_{x,y \in X} g_{k,l}(x^*y) = \sum_{x \in X}\int_{\Om(d)} g_{k,l}(x^*y) \dd y = \sum_{x \in X} \ip{1}{g_{k,l,x}} = 0.
\]
Similarly, consider the $(i,j)$ entry of $H_{k,l}^*H_{k',l'}$, namely
\[
(H_{k,l}^*H_{k',l'})_{ij} = \sum_{x \in X} \overline{e_i(x)}e_j(x),
\]
where $e_i \in \Hom(k,l)$ and $e_j \in \Hom(k',l')$, and compare to 
\[
|X| \int_{\Om(d)} \overline{e_i(x)}e_j(x) \dd x = |X| \de_{k,k'} \de_{l,l'} \de_{i,j}.
\]
All entries are equal if and only if $X$ is a $\cT$-design.
\end{proof}

From Lemma~\ref{lem:designchar} it follows that for a complex spherical $\cT$-design $X$ and any element $\sigma$ of unitary group, the image of X under the element of $\sigma$ is also a complex spherical $\cT$-design. 

Spherical designs are in some sense optimal with respect to taking polynomials of inner products, as the following lemma shows. 
The following inequalities can be regarded as a complex analogue of the result in \cite{sid}.

\begin{lemma}\label{lem:moment}
Let $X$ be a finite nonempty subset of $\Om(d)$. For any $(k,l) \in \nn^2$, 
\begin{equation}\label{eqn:polysum}
\frac{1}{|X|^2} \sum_{x,y \in X} (x^*y)^k(y^*x)^l \geq \begin{cases}
{d+k-1 \choose k-1}^{-1} & \mbox{if }k = l; \\
0& \mbox{otherwise.}
\end{cases}
\end{equation}
Moreover, let $\cT$ be a lower set. Then equality holds in \eqref{eqn:polysum} for all $(k,l) \in \cT$ if and only if $X$ is a $\cT$-design.
\end{lemma}

\begin{proof} Consider
\[
S_{k,l} := \frac{1}{|X|} \sum_{x \in X} x^{\otimes k}\otimes \overline{x}^{\otimes l} - \int_{\Om(d)} x^{\otimes k}\otimes \overline{x}^{\otimes l} \dd x.
\]
Then 
\[
S_{k,l}^*S_{k,l} = \frac{1}{|X|^2} \sum_{x,y \in X} (x^*y)^k(y^*x)^l -
\int_{\Om(d)} (x^*a)^k(a^*x)^l \dd x \geq 0,
\]
for $a \in \Om(d)$. It follows that
\[
\frac{1}{|X|^2} \sum_{x,y \in X} (x^*y)^k(y^*x)^l \geq \int_{\Om(d)} (x^*e_1)^k(e_1^*x)^l \dd x = \de_{k,l} {d+k-1 \choose k-1}^{-1}.
\]
Equality holds if and only if $S_{k,l} = 0$, which occurs if and only if $X$ is a $(k,l)$-regular set. 
\end{proof}

\subsection*{Real spherical and complex projective designs}

Complex spherical $\cT$-designs are closely related to both complex projective designs and real spherical designs. Given a finite set of points $L$ on a complex unit sphere, we let $P(L)$ denote the corresponding set of points in projective space (that is, the $1$-dimensional subspaces spanned by the unit vectors). Assume the points in $L$ span distinct $1$-dimensional subspaces. Then $P(L)$ is a \defn{complex projective $t$-design} if $L$ is $(t,t)$-regular~\cite{dgs2}.

\begin{lemma}\label{lem:projectivedesign}
Let $L$ be a set of points in the complex unit sphere such that $\abs{x^*y} < 1$ for all $x,y \in L$, and let $X$ be an $n$-antipodal cover of $L$. 
\begin{enumerate}[(i)]
\item 
If $X$ is a complex spherical $\cT$-design, then $P(L)$ is a projective $t$-design, where $t = \max\{k : (k,k) \in \cT\}$. 
\item If $P(L)$ is a projective $t$-design, and $n > 1$ is coprime to $\{2,3,\ldots,t\}$, then $X$ is a spherical $\cT$-design where $\cT = \cl(\{(t,t)\})$.
\end{enumerate}
\end{lemma}

\begin{proof}
\begin{enumerate}[(i)]
\item
Let $\om$ be an $n$-th root of unity. Then for every $f \in \Hom(t,t)$ and $x \in \Om(d)$, $f(\om x) = f(x)$. Therefore, since $X$ is $(t,t)$-regular, so is $L$, and so $P(L)$ is a projective $t$-design.
\item 
If $P(L)$ is a $t$-design, then $L$ is $(t,t)$-regular and so is $X$. Thus $X$ is $(k,k)$-regular for every $0 \leq k \leq t$. Next, suppose $k \neq l$ and $0 \leq k,l \leq t$. Since $X$ is $n$-antipodal and $n$ does not divide $k-l$, Lemma \ref{lem:avgantipodal} implies that $X$ is $(k,l)$ regular. So $X$ is $(k,l)$ regular for all $k,l \leq t$ and is therefore a $\cl\{(t,t)\}$-design.
\end{enumerate}
\end{proof}
 
\begin{example}
Let $L$ be the standard basis for $\cx^d$. Since $P(L)$ is a complex projective $1$-design, $L$ is $(1,1)$-regular and $(0,0)$-regular. However, $L$ is not $(1,0)$-regular. Letting $X = L \cup -L$, we find that $X$ is $(1,0)$-regular (and $(0,1)$-regular), so $X$ is a $\cl\{(1,1)\}$-design.
\end{example}

Lemma \ref{lem:projectivedesign} emphasizes the main difference between projective and spherical designs. In a projective design, the points are $1$-dimensional subspaces rather than unit vectors, so it makes little sense to consider two different vectors spanning the same 1-dimensional space. In later sections, we will give examples of spherical designs which are not simply covers of projective designs.

Let $S^{d-1}$ denote the real unit sphere in $\re^d$.
A \defn{real spherical $t$-design} is a finite set of points $Y \sbs S^{d-1}$ such that any polynomial of degree at most $t$ in the coordinates of $x \in \re^d$ has the same average over $Y$ as it does over $S^{d-1}$.  
Like complex spherical designs, real designs may be characterized using orthogonal polynomials, called \defn{Gegenbauer polynomials} $Q_{d,k}(x)$ defined recursively as follows;
\begin{align}\label{Gegenbauer}
Q_{d,0}(x)&:=1,\quad Q_{d,1}(x):=dx,\\
 Q_{d,k+1}(x)&:=\tfrac{1}{\lambda_{d,k+1}}(xQ_{d,k}(x)+(\lambda_{d,k-1}-1)Q_{d,k-1}(x)),
\end{align}
where $\lambda_{d,k}=\tfrac{k}{d+2k-2}$.
A finite set $Y \sbs S^{d-1}$ is a real spherical $t$-design if and only if, for every $0 \leq k \leq t$, 
\[
\sum_{x,y \in Y} Q_{d,i}(x^Ty) = 0.
\]
For even dimensions, the Gegenbauer polynomials satisfy
\begin{equation}\label{eqn:poly}
Q_{2d,k}(\Re(x)) = \sum_{i=0}^k g_{d,i,k-i}(x).
\end{equation}
Delsarte et al.~\cite{dgs} proved that if $X$ is a real spherical $t$-design in $S^{d-1}$ for a positive integer $t$, then 
\begin{equation}\label{eqn:realbound}
|X| \geq \binom{d+t/2-1}{t/2}+\binom{d+t/2-2}{t/2-1},
\end{equation}
if $t$ is even, and
\begin{equation}\label{eqn:realbound2}
|X| \geq 2\binom{d+(t-1)/2-1}{(t-1)/2},
\end{equation}
if $t$ is odd.
If equality holds, then the angles $\{x^Ty : x, y \in X, x \neq \pm y\}$ are all roots of the polynomial $Q_{d,\lfloor t/2\rfloor}(x)$.

To relate real and complex spherical designs, define the following map $\phi: \cx^d \rightarrow \re^{2d}$:
\begin{equation}\label{eqn:phi}
\phi(x_1,\ldots,x_d) = (\Re(x_1),\Im(x_1),\ldots,\Re(x_d),\Im(x_d)).
\end{equation}
For $x$ and $y$ in $\cx^d$, $\phi(x)^T\phi(y) = \Re(x^*y)$. It follows that $\phi$ maps points in $\Om(d)$ to points in $S^{2d-1}$.

\begin{lemma}\label{lem:realdesign}
Let $X$ be a complex spherical $\cT$-design in $\Om(d)$, and $t$ a positive integer.
The following are equivalent.
\begin{itemize}
\item[(i)] A lower set $\cT$ contains $\{(k,l) \in \nn^2 : k+l \leq t\}$.
\item[(ii)] The set $\phi(X)$ is a real spherical $t$-design in $S^{2d-1}$.
\end{itemize}
\end{lemma}

\begin{proof}
Since $X$ is a $\cT$-design, by Lemma \ref{lem:designchar} we have $\sum_{x,y \in X} g_{d,k,l}(x^*y) = 0$ for every $k$ and $l$ with $(k,l) \in \cT$. So, for a positive integer $i$,
\begin{align*}
\sum_{x,y\in X} Q_{2d,i}(\phi(x)^T\phi(y))=0
\Leftrightarrow &  \sum_{x,y\in X} Q_{2d,i} (\Re(x^*y))=0 \\
\Leftrightarrow &  \sum_{k=0}^i \sum_{x,y\in X} g_{d,k,i-k}(x^*y)=0 \\
\Leftrightarrow & \sum_{x,y\in X} g_{d,k,i-k}(x^*y)= 0 \text{ for } 0\leq k\leq i.
\end{align*}
From characterizations of real and complex spherical designs, (i) and (ii) are equivalent. 
\end{proof}

\section{Complex spherical codes and absolute bounds}\label{sec:complexcode}

In this section, we introduce complex spherical codes, which are in some ways dual to complex spherical designs, and give some simple bounds on the size of both. 

Given $X \sbs \Om(d)$, we define the \defn{inner product set} of $X$ to be 
\[
A(X) := \{a^*b : a, b \in X, a \neq b \}.
\]
A polynomial $F(x)\in\re[x,\co{x}]$ is said to be an \defn{annihilator polynomial} of $X$ if 
$F(\al)=0$ for each $\al\in A(X)$ and $F(1)$ is positive. 

\begin{definition}
We say $X$ is a \defn{complex spherical code of degree $s$} if $|A(X)| = s$.
For a lower set $\cS$ in $\nn^2$,
a finite subset $X$ in $\Om(d)$ is said to be an \defn{$\cS$-code} if $X$ has an annihilator polynomial in the span of $\{x^k\co{x}^l: (k,l) \in \cS\}$.
\end{definition}
Let $X$ be a finite set in $\Om(d)$ with inner product set $A(X)=\{\al_1,\ldots,\al_s\}$.
Set $\al_0=1$.
For $0 \leq i \leq s$, we let $A_i$ denote the $(0,1)$-matrix whose rows and columns are indexed by $X$ such that $(A_i)_{xy}=1$ if $x^*y=\al_i$ and $(A_i)_{xy}=0$ otherwise.

If $X$ is a complex spherical code of degree $s$, then $F(x)=\prod_{\al \in A(X)}(x-\al)$ is an annihilator polynomial of $X$ and consequently $X$ is an $\cS$-code with $\cS=\cl\{(s,0)\}$.
However, in many cases, we can choose a smaller lower set $\cS$ depending on the elements of $A(X)$ For example, if each of the $s$ elements of $A(X)$ have the same absolute value $|\al|$, then $F(X) = x\bar{x} - |\al|^2$ is an annihilator polynomial and $X$ is an $\cS$-code with $\cS=\cl\{(1,1)\}$.

Complex spherical codes are closely related to other types of codes. A \defn{real spherical code of degree $s$} is a set $X$ of points on the real unit sphere whose inner product set $\{a^Tb : a, b \in X, a \neq b \}$ has size $s$ (see \cite{dgs2}). If $X$ is a complex spherical code of degree $s$ in $\Om(d)$, and $\phi: \cx^d \rightarrow \re^{2d}$ is the function in \eqref{eqn:phi}, then $\phi(X)$ is a real spherical code of degree at most $s$. Conversely, if $X$ is a real spherical code of degree $s$ in $\re^d$, then it is also complex spherical code of degree $s$ in $\Om(d)$ under the natural embedding $\re^d \hookrightarrow \cx^d$.

Likewise, a \defn{complex projective code of degree $s$} is a set $P$ of projection matrices for one-dimensional subspaces in $\cx^d$ whose inner product set $\{\tr(P_xP_y): P_x,P_y \in P, x \neq y\}$ has size $s$ \cite{dgs}. If $P(X)$ denotes the set of $1$-dimensional complex subspaces spanned by the vectors in $X \sbs \Om(d)$, then $P(X)$ is a complex projective code of degree at most $s$.

We can now give simple lower bounds on the size of a $\cT$-design and upper bounds on the size of on $\cS$-code. These ``absolute" bounds (Theorem~\ref{thm:absbound}) do not depend the values in $A(X)$, unlike the tighter linear programming bounds we will see in the next section. For collections of indices $\cU,\cV\sbs\nn^2$, we define the convolution of $\cU$ and $\cV$ as follows:
\[
\cU*\cV := \{(k+l',k'+l): (k,l)\in\cU,(k',l')\in\cV\}.
\] 

\begin{theorem}\label{thm:absbound}
\begin{itemize}
\item[(i)] If $X$ is a $\cT$-design, and there exists a lower set $\cU \sbs \cT$ such that $\cU * \cU \sbs \cT$, then
\[
|X| \geq \sum_{(k,l) \in \cU} \dim(\Harm(k,l)).
\]
\item[(ii)] If $X$ is an $\cS$-code, then 
\[
|X| \leq \sum_{(k,l)\in\cS}\dim(\Harm(k,l)).
\]
\end{itemize}
\end{theorem}

\begin{proof}
\begin{itemize}
\item[(i)] Let $S_1,\ldots,S_N$ be an orthonormal basis for $\bigoplus_{(k,l) \in \cU} \Harm(k,l)$. Then $\overline{S_i}S_j$ is in $\bigoplus_{(k,l) \in \cT} \Harm(k,l)$, and so 
\[
\frac{1}{|X|}\sum_{x \in X} \overline{S_i(x)}S_j(x) = \ip{1}{\overline{S_i}S_j} = \ip{S_i}{S_j} = \de_{i,j}.
\]
But note that 
\[
\ip{f}{g}_X := \frac{1}{|X|}\sum_{x \in X} \overline{f(x)}g(x)
\]
defines an inner product on functions on $X$, so $S_i$ and $S_j$ are orthogonal as functions on $X$. Therefore the dimension of the space of functions on $X$ is at least the dimension of $\bigoplus_{(k,l) \in \cU} \Harm(k,l)$.
\item[(ii)] Let $F(x) \in \spn\{x^k\co{x}^l: (k,l) \in \cS\}$ be the annihilator of $X$. For a point $a \in X$, define a function $f_a: \Om(d) \rightarrow \cx$ called the \defn{annihilator of $X$ at  $a$} as follows:
\[
f_a(z) = F(a^*z).
\]
Since $f_a(a) \neq 0$ and $f_a(b) = 0$ for all $b \neq a$ in $X$, it follows that $\{f_a : a \in X\}$ is a linearly independent set of functions on $X$. Moreover, since each $f_a$ is in $\bigoplus_{(k,l)\in\cS}\Harm(k,l)$, the size of the set (namely $|X|$) is at most the dimension of the space $\bigoplus_{(k,l)\in\cS}\Harm(k,l)$.
\end{itemize}
\end{proof}

\begin{corollary}
If $X \sbs \Om(d)$ has degree $s$, then
\[
|X| \leq \dim(\Hom(s,0)) = {d+s-1 \choose d-1}.
\]
\end{corollary}

\section{LP Bounds}\label{sec:bounds}

In this section we show how the linear programming technique of Delsarte~\cite{del} applies to complex spherical designs and codes. The resulting ``relative" bounds  (Theorem~\ref{thm:relbound}) on the size of $X$ depend the values in $A(X)$ and are generally better than the absolute bounds in Theorem~\ref{thm:absbound}. We also show that in the case of equality in Theorem~\ref{thm:absbound}, codes and designs coincide, and give examples of tightness.

Recall that $H_{k,l}$ denotes the $(k,l)$-characteristic matrix of $X$ defined in Section \ref{sec:designs}. 

\begin{lemma}\label{polymatrix}
Let $X$ be a finite nonempty subset in $\Om(d)$ with inner product set $A(X)=\{\al_1,\ldots,\al_s\}$.
For any polynomial $F(x)=\sum_{(k,l)\in\nn^2}f_{k,l}g_{k,l}(x)$, 
$$f_{0,0}\abs{X}^2+\sum_{(k,l)\in\nn^2\setminus\{(0,0)\}}f_{k,l}||H_{k,l}^*H_{0,0}||=F(1)|X| + \sum_{i=1}^sF(\al_i)d_{i},$$
where $d_{i}=|\{(x,y)\in X^2:\ip{x}{y}=\al_i\}|$. 
\end{lemma}

\begin{proof}
From the Addition Formula in Theorem~\ref{thm:addition} we get the following equation for each $(k,l) \in \nn^2$:
\[
H_{k,l}H_{k,l}^*=\sum_{i=0}^sg_{k,l}(\al_i)A_{i}.
\]
Summing over all $(k,l)$,
\begin{equation}\label{eqn:polymatrix}
\sum_{(k,l)\in\nn^2}f_{k,l}H_{k,l}H_{k,l}^*=F(1)I +\sum_{i=1}^sF(\al_i)A_{i}.
\end{equation}
Taking the sum of all entries in the above equation, we obtain
\begin{align*}
F(1)|X| + \sum_{i=1}^sF(\al_i)d_{i}
&=\sum_{(k,l)\in\nn^2}f_{k,l}H_{0,0}^*H_{k,l}H_{k,l}^*H_{0,0} \\
&=\sum_{(k,l)\in\nn^2}f_{k,l}||H_{k,l}^*H_{0,0}|| \\
&=f_{0,0}|X|^2+\sum_{(k,l)\in\nn^2\setminus\{(0,0)\}}f_{k,l}||H_{k,l}^*H_{0,0}||.
\end{align*}
\end{proof}

As two special cases, note that $F(\al) = 0$ if $F(x)$ is an annihilator of $X$, while $||H_{k,l}^*H_{0,0}|| = 0$ if $X$ is a $\cT$-design with $(k,l) \neq (0,0)$ in $\cT$.

\begin{theorem}\label{thm:relbound}
Suppose $X$ is a $\cT$-design and $F(x) = \sum_{k,l} f_{k,l} g_{k,l}(x)$ is a bivariate polynomial such that $f_{0,0} > 0$.
\begin{itemize}
\item[(i)] If $f_{k,l} \leq 0$ whenever $(k,l)$ is not in $\cT$, and $F(\al) \geq 0$ for every $\al \in A(X)$, then
\[
|X| \geq \frac{F(1)}{f_{0,0}}.
\]
\item[(ii)] If $f_{k,l} \geq 0$ for all $k$ and $l$, and $F(\al) \leq 0$ for every $\al \in A(X)$, then
\[
|X| \leq \frac{F(1)}{f_{0,0}}.
\]
\end{itemize}
\end{theorem}

\begin{proof}
\begin{itemize}
\item[(i)] If $(k,l) \neq (0,0)$ is in $\cT$, then $||H_{k,l}^*H_{0,0}|| = 0$. Otherwise, $f_{k,l} \leq 0$. In either case, $f_{k,l}||H_{k,l}^*H_{0,0}|| \leq 0$. Combining this inequality with $F(\al) \geq 0$ for all $\al \in A(X)$, the formula in Lemma \ref{polymatrix} reduces to 
 $$f_{0,0}\abs{X}^2 \geq F(1)\abs{X},$$
from which the result follows.
\item[(ii)] If $f_{k,l} \geq 0$ for all $(k,l)$, then $f_{k,l}||H_{k,l}^*H_{0,0}|| \geq 0$. Combining this inequality with $F(\al) \leq 0$ for all $\al \in A(X)$, the formula in Lemma \ref{polymatrix} reduces to 
 $$f_{0,0}\abs{X}^2 \leq F(1)\abs{X}.$$
\end{itemize}
Note that if equality holds in either case, then $F(\al) = 0$ for all $\al \in A(X)$ and $f_{k,l}||H_{k,l}^*H_{0,0}|| = 0$ for all $(k,l) \neq (0,0)$.
\end{proof}

In fact, there is a slightly more general version of Theorem~\ref{thm:relbound}(ii) which will be useful in investigation tightness of equalities.

\begin{lemma}\label{lem:coefbound}
Let $X$ be an $\cS$-code in $\Om(d)$ with an annihilator polynomial $F(x)=\sum_{(k,l)\in\cS}f_{k,l}g_{k,l}(x)$ such that all $f_{k,l}$ are positive. Then for each $(k,l)\in\cS$,
\[
\abs{X} \leq F(1)/f_{k,l}.
\] 
\end{lemma}

\begin{proof}
For any $(m,n)\in\cS$, define the annihilator polynomial 
\[
H(x)=F(x)g_{n,m}(x)/g_{n,m}(1).
\]
Lemma~\ref{coef} implies that $h_{0,0}=f_{m,n}$ and $h_{k,l}\geq0$ for all $(k,l)$.
Hence Theorem~\ref{thm:relbound}(ii) gives $F(1)-\abs{X}f_{m,n}=F(1)-\abs{X}h_{0,0}\geq0$.
\end{proof}

Recall that from Theorem~\ref{thm:absbound} we have the following bounds: 
for a $\cT$-design and $\cS$-code $X$ and a lower set $\cU$ such that $\cU*\cU\sbs\cT$,
$$\sum_{(k,l)\in\cU}\dim(\Harm(k,l))\leq |X|\leq\sum_{(k,l)\in\cS}\dim(\Harm(k,l)).$$

We say that $X$ is a \defn{tight design with respect to $\cU$} if $X$ is a $\cU*\cU$-design and attains equality in the absolute bound for $\cU*\cU$-designs in Theorem~\ref{thm:absbound}(i). 
Similarly, an $\cS$-code $X$ is \defn{tight} if $X$ attains the bound for $\cS$-codes in Theorem~\ref{thm:absbound}(ii).
The definition of tightness for $\cT$-designs with respect to $\cU$ seems to be a complex analogue for tightness of real spherical even designs in a sense. The following theorem shows the equivalence of tightness for designs and codes. 

\begin{theorem}\label{thm:equiv}
Let $X$ be a finite nonempty subset of $\Om(d)$ and let $\cS$ be a lower set. Then the following are equivalent:
\begin{itemize}
\item[(i)] $X$ is an $\cS$-code and a $\cS * \cS$-design.
\item[(ii)] $X$ is a tight $\cS$-code.
\item[(iii)] $X$ is a tight design with respect to $\cS$.
\end{itemize}
\end{theorem}

\begin{proof}
\begin{itemize}
\item[(i)] 
Recall that we use $m_{k,l}$ to denote $\dim(\Harm(k,l))$.

If $X$ is an $\cS$-code, then $|X| \leq \sum_{(k,l)\in\cS}m_{k,l}$. If $X$ is a $\cS*\cS$-design, then $|X| \geq \sum_{(k,l)\in\cS}m_{k,l}$. Combining the two, $|X| = \sum_{(k,l)\in\cS}m_{k,l}$. So $X$ is both a tight $\cS$-code and a tight design with respect to $\cS$.

\item[(ii)] Assume $X$ is a tight $\cS$-code, so that $|X|=\sum_{(k,l)\in\cS}m_{k,l}$. If $F(x)=\sum_{(k,l)\in\cS}f_{k,l}g_{k,l}(x)$ is the annihilator polynomial of $X$, then as in equation~\eqref{eqn:polymatrix},
$$\sum_{(k,l)\in\cS}f_{k,l}H_{k,l}H_{k,l}^*=F(1)I.$$
Define a matrix $H=[H_{k,l}]_{(k,l)\in\cS}$, which is a square matrix of size $|X|$. The left-hand side of the previous equation is $HBH^*$, where $B=\oplus_{(k,l)\in\cS}f_{k,l}I_{k,l}$ and $I_{k,l}$ is the identity matrix of size $m_{k,l}$. Hence $H$ is nonsingular, and comparing the signature implies that all $f_{k,l}$ are positive. 
Lemma~\ref{lem:coefbound} shows $f_{k,l}\leq F(1)/|X|$ for each $(k,l)\in\cS$.
Then the inequality
\begin{align*}
F(1) = \sum_{(k,l)\in\cS}f_{k,l}g_{k,l}(1) \leq \frac{F(1)}{|X|} \sum_{(k,l)\in\cS}g_{k,l}(1) = F(1)
\end{align*}
yields $f_{k,l}=F(1)/|X|$ for each $(k,l)\in\cS$. Normalizing so that $f_{k,l} = 1$, we have that $F(x) = \sum_{(k,l)\in\cS}g_{k,l}(x)$ is an annihilator for $X$. 

Consider instead $\hat{F}(x) =(\sum_{(k,l)\in \cS}g_{k,l}(x))(\sum_{(k,l)\in \cS}g_{l,k}(x))$, also an annihilator for $X$. Then $\hat{F}(1)=(\sum_{(k,l)\in \cS}m_{k,l})^2$ and $\hat{F}(\alpha) = 0$ for any $\alpha\in A(X)$. Using the Jacobi polynomial expansion $\hat{F}(x)=\sum_{k,l}\hat{f}_{k,l}g_{k,l}(x)$, we see that $\hat{f}_{k,l} > 0$ for all $(k,l) \neq (0,0)$ in $\cS*\cS$ and $\hat{f}_{0,0}=\sum_{(k,l)\in \cS}m_{k,l}$ by Lemma \ref{coef}. Therefore by Theorem~\ref{thm:relbound}(ii), $|X|\leq \hat{F}(1)/\hat{f}_{0,0}=\sum_{(k,l)\in \cS}m_{k,l}$.

Since equality holds, from the proof of Theorem~\ref{thm:relbound}(ii) it follows that $\hat{f}_{k,l}||H_{k,l}^*H_{0,0}|| = 0$ for all $(k,l) \neq (0,0)$ in $\cS*\cS$. But $\hat{f}_{k,l} > 0$, so $||H_{k,l}^*H_{0,0}|| = 0$ and $X$ is a (tight) $\cS*\cS$-design.

\item[(iii)] Assume $X$ is a tight design with respect to $\cS$, so  $\abs{X} = \sum_{(k,l)\in\cS}m_{k,l}$. Again consider $\hat{F}(x)=(\sum_{(k,l)\in \cS}g_{k,l}(x))(\sum_{(k,l)\in \cS}g_{l,k}(x))$. Then $\hat{F}(\alpha)\geq0$ for any $\alpha\in A(X)$, and the coefficients in the Jacobi polynomial expansion satisfy $\hat{f}_{k,l} = 0$ for $(k,l) \notin \cS*\cS$ and $\hat{f}_{0,0}=\sum_{(k,l)\in \cS}m_{k,l}$. Therefore by Theorem~\ref{thm:relbound}(ii), $|X|\geq \hat{F}(1)/\hat{f}_{0,0}=\sum_{(k,l)\in \cS}m_{k,l}$.

Since equality holds, from the proof of Theorem~\ref{thm:relbound} we see that $\hat{F}(\al) = 0$ for all $\al \in A(X)$. Therefore $\sum_{(k,l)\in \cS}g_{k,l}(\al) = 0$ and so $\sum_{(k,l)\in \cS}g_{k,l}$ is an annihilator for $X$. Thus $X$ is a (tight) $\cS$-code. 
\end{itemize}
\end{proof}

The following examples show that the relative bound is sometimes tight. 
\begin{example}{\bf(SIC-POVMs)} \label{ex:sicpovm}
Let $X$ be a finite subset in $\Omega(d)$ with the inner product set 
$A(X)=\{\pm\tfrac{1}{\sqrt{d+1}},\pm\tfrac{\cxi}{\sqrt{d+1}},\pm\cxi,-1\}$.
Taking 
\[
F(x)=\frac{d((d+1)x\bar{x}-1)(x^2+\bar{x}^2+2x+2\bar{x}+2)}{2},
\] 
we have $F(1)=4d^2$ and $F(\alpha)=0$ for all $\alpha\in A(X)$. Letting $F(x)=\sum_{k,l=0}^3f_{k,l}g_{k,l}(x)$, we find $f_{0,0} = 1$, and so by Theorem~\ref{thm:relbound}, $|X| \leq 4d^2$. Examples of such sets come from quantum information, where they are known as \defn{SIC-POVMs} \cite{rbsc}. Define the \defn{Pauli operators}
\[
P_x = \left(\begin{matrix} 0 & 1\\1 & 0 \end{matrix} \right), \quad P_z = \left(\begin{matrix} 1 & 0\\0 & -1 \end{matrix} \right).
\]
For $d = 2$, let $X$ be the orbit of $v = (1,\frac{-1-\sqrt{3}}{2}(1+\cxi))$ under the group generated by $\{P_x ,P_z, \cxi I\}$. For $d = 8$, let $X$ be the orbit of $v = (0,0,1+\cxi,1-\cxi,1+\cxi,-1-\cxi,0,2)/\sqrt{2}$ under the group generated by $\{P_x ,P_z, \cxi I\}^{\otimes 3}$ (this construction is due to Hoggar \cite{hoggar}). In both cases $|X| = 4d^2$. The corresponding points $P(X)$ in projective space also satisfy the projective Delsarte bound: $\abs{P(X)} = d^2$.
\end{example}

\begin{example}{\bf(MUBs and Kerdock codes)}
Let $X$ be a finite subset in $\Omega(d)$ with the inner product set 
$A(X)=\{\frac{\pm1\pm\cxi}{\sqrt{2d}},0,\pm\cxi,-1\}$. Taking 
\[
F(x)=d(x^4+\bar{x}^4)+2d(x^3+\bar{x}^3)+(d+1)(x^2+\bar{x}^2)+2(x\bar{x}+x+\bar{x}),
\] 
we find that $F(x)$ is an annihilator for $A(X)$, and so applying the relative bound,
\[
|X|\leq 4d(d+1).
\]
Similarly, if $A(X)=\{\frac{\pm1}{\sqrt{d}},\frac{\pm\cxi}{\sqrt{d}},0,\pm\cxi,-1\}$, then
\[
F(x)=d(d+1)(dx\bar{x}-1)(x+\bar{x})(x+\bar{x}+1)
\]
results in a bound of $|X| \leq 4d(d+1)$. 
Both of these bounds can be obtained using \defn{ $\zz_4$-Kerdock codes}~\cite{hkcss,hp}.

Let $C$ be a $\zz_4$-linear error-correcting code of length $d$ which contains the all-ones vector $\one$. For each $x = (x_1,\ldots,x_d) \in C$, define
\[
\hat{x} := \frac{1}{\sqrt{d}}(\cxi^{x_1},\ldots,\cxi^{x_d}) \in \Om(d),
\]
and consider $\hat{C} := \{\hat{x} : x\in C\}$. For fixed $x,y \in C$, let $n_j$ ($0 \leq j \leq 3$) denote the number of entries of $y-x \in C$ equal to $j$. Then 
\[
\ip{\hat{x}}{\hat{y}} = \frac{1}{d}[n_0 - n_2 + \cxi(n_1-n_3)].
\]
That is, the inner products in $\hat{C}$ depend only on the weights $(n_j)$ of the elements of $C$. 

Now let $C = \hat{K}(r+1)$, the $\zz_4$-linear Kerdock code of length $d = 2^r$, $r$ odd. Following \cite{hp}, there are codewords in $C$ with the following weights (plus the rotations of these weights obtained by adding $\one$):
\begin{align*}
(n_0,n_1,n_2,n_3) = & (2^r,0,0,0), \\ 
& (2^{r-1},0,2^{r-1},0), \\
& (2^{r-2} + \de_12^{\frac{r-3}{2}}, 2^{r-2} + \de_22^{\frac{r-3}{2}}, 2^{r-2} - \de_12^{\frac{r-3}{2}}, 2^{r-2} - \de_22^{\frac{r-3}{2}}),
\end{align*}
where $\de_i = \pm 1$. From these values of $(n_j)$ we get angles $1$, $0$, and $ \frac{1}{\sqrt{2d}}(\de_1 + \cxi \de_2)$ respectively, plus their rotations by $\cxi$. Finally, let $X := \hat{C} \cup \{\tfrac{\pm 1 \pm \cxi}{\sqrt{2}}e_j :  1\leq j \leq d\}$ be our spherical code in $\Om(d)$, where $e_j$ is a standard basis vector. Then $X$ has angle set $A(X)=\{\frac{\pm1\pm\cxi}{\sqrt{2d}},0,\pm\cxi,-1\}$ and has size $\abs{X} = 4d(d+1)$. The construction is similar for $r$ even.

\end{example}
For $d\leq 3$, tight even designs $S^{d-1}$ exist only if $t=2,4$ see \cite{bd1,bd2,bmv}.
\begin{example}{\bf(Tight even real designs)}
Let $d$ be a integer at least $2$ and $t$ be $2$ or $4$.
Consider $\cU= \{(k,l)\in\nn^2:k+l\leq t/2\}$. 
Let $X$ be a subset of $\Om(d)$ such that $\phi(X)$ is a tight $t$-design in $S^{2d-1}$ (where $\phi$ is the natural embedding of $\cx^d \rightarrow \re^{2d}$ given in equation~\eqref{eqn:phi}).
Here, if there exists a tight $4$-design in $S^{d-1}$, then $d$ must be $2$ or $(2m+1)^2-1$ for some integer $m$, in particular $d$ is an even integer, see \cite{bd1,bd2}.
Then $X$ is a tight design in $\Om(d)$ with respect to $\cU$. 
\end{example}

\subsection*{Bounds on $n$-antipodal codes}
We can show bounds on the size of $n$-antipodal codes in a similar way to Theorem~\ref{thm:absbound}.
Moreover, if equality holds then $n$-antipodal codes become complex spherical designs.

For a lower set $\cS$ and a positive integer $n\geq2$, 
define $\cS_n=\{(k,l)\in \cS: k\equiv l \mod n\}$. 
Then we let $\tilde{\cS}$ denote 
For a subset $\cU\sbs\nn^2$, 
a nonempty subset $X$ in $\Om(d)$ is said to be \defn{$\cU$-regular} 
if $X$ is $(k,l)$-regular for all $(k,l)\in\cU$.
Define $\mathcal{L}_n$ to be a set of indices $(j,0)$ and $(0,j)$ 
where $0\leq j\leq \lfloor n/2\rfloor$.
\begin{theorem}\label{thm:antipodal}
Let $\cS$ be a lower set and $n$ a positive integer at least $2$. 
Let $X$ be an $\cS$-code with an annihilator polynomial $F(x)$ in the span of $\{x^k\co{x}^l: (k,l) \in \cS\}$ satisfying the following condition:
\begin{align}\label{antipodalpoly}
F(w^j\al)=0 \text{ for any } \al \in A^*(X),0 \leq j \leq n-1. 
\end{align}
Then $|X|\leq n\min\{\sum_{(k,l)\in \cS_n}\dim(\Harm(k,l)),\sum_{(k,l)\in \cS\setminus\cS_n}\dim(\Harm(k,l))\}$. 
Moreover, if equality holds, then $X$ is $n$-antipodal and $\cT$-design where $\cT$ is the maximal lower set contained in $\mathcal{L}_n*(\widehat{\cS}*\widehat{\cS})$,  $\widehat{\cS}$ is defined as follows:
$\widehat{\cS}$ is $\cS_n$ if $|X|=n\sum_{(k,l)\in \cS_n}\dim(\Harm(k,l))$, $\cS\setminus \cS_n$ if $|X|=n\sum_{(k,l)\in \cS\setminus S_n}\dim(\Harm(k,l))$.
\end{theorem}    
\begin{proof}
For $x\in \Om(d)$, set $S_x=\{w^jx:0\leq j\leq n-1\}$.
Define $L$ to be a set of representatives of $S_x$ for every $x\in X$. 
Then $|X|\leq n|L|$, and equality holds if and only if $X$ is $n$-antipodal. 
Set $F_n(x)=\tfrac{1}{n}\sum_{j=0}^{n-1}F(w^jx)$.
Since $A(L)$ is contained in $A^*(X)$, $F_n(x)$ is an annihilator polynomial of $L$. 
Since $\sum_{j=0}^{n-1}g_{k,l}(w^jx)=0$ provided that $k\not\equiv l\mod n$, $F_n(x)$ is in the span of $\{x^k\co{x}^l: (k,l) \in \cS_n\}$.
While, $F(x)-F_n(x)$ is an annihilator polynomial in the span of $\{x^k\co{x}^l: (k,l) \in \cS\setminus \cS_n\}$ 
Then the same way in Theorem~\ref{thm:absbound}(ii) shows 
$$|L|\leq \min\{\sum_{(k,l)\in \cS_n}\dim(\Harm(k,l)),\sum_{(k,l)\in \cS\setminus\cS_n}\dim(\Harm(k,l))\}.$$
Thus we have the desired bound on the size of $X$. 

Assume $X$ attains the bound above, namely $X$ is an $n$-antipodal and  $|L|=\sum_{(k,l)\in\widehat{\cS}}m_{k,l}$. 
As is the same way in Theorem~\ref{thm:equiv}, 
we have that $\widehat{F}(x) = \sum_{(k,l)\in\widehat{\cS}}g_{k,l}(x)$ is an annihilator for $L$, 
and $L$ is $\widehat{\cS}*\widehat{\cS}$-regular.
Since $X$ is the $n$-antipodal cover of $L$, thus $X$ is a $\cT$-design for the desired $\cT$.
\end{proof}
\begin{remark}
If $\cS_n$ lies in $\{(k,k)\in\nn^2:0\leq k\}$, then the bound given in Theorem~\ref{thm:antipodal} coincides with the bound as the projective code.
\end{remark}

For $3\leq d$, tight odd designs $S^{d-1}$ exist only if $t=1,3,5,7,11$ see \cite{bd1,bd2,bmv}.
Tight $1$-designs are $\{x,-x \}$ for any $x\in S^{d-1}$, and 
tight $3$-designs are cross polytopes in $S^{d-1}$, namely $\{\pm f_1,\ldots,\pm f_d\}$ for any orthonormal basis $\{f_1,\ldots,f_d\}$ in $\re^d$. 
The existence of tight $5$-designs in $S^{d-1}$ is equivalent to that of tight $4$-designs in $S^{d-2}$,
so tight $5$-designs exists in $S^{d-1}$ for some odd integer $d$.
Tight $7$-design in even dimension exists for $d=8$, that is the $E_8$ root system, 
and the only tight $11$-design is the minimum vectors of the Leech lattice in $\re^{24}$.
\begin{example}{\bf(Tight odd real designs)}
Let $d$ be a integer at least $2$ and $t$ be $1,3,7$ or $11$.
Let $X$ be a subset of $\Om(d)$ such that $\phi(X)$ is a tight $t$-design in $S^{2d-1}$ (where $\phi$ is the natural embedding of $\cx^d \rightarrow \re^{2d}$ given in equation~\eqref{eqn:phi}).
Then $X$ attains the bound in Theorem~\ref{thm:antipodal} for $n=2$. 
\end{example}

\section{Association schemes}\label{sec:schemes}

In this section we consider complex spherical designs whose inner product relations carry the structure of a nonsymmetric association scheme.
In contrast with real spherical designs or projective designs, not every tight complex spherical code (or tight complex spherical design) gives rise to a scheme. Nevertheless, we do get a scheme when the strength of the design is high compared to its degree.

Let $X$ be a nonempty finite set and let $R_i$ be a nonempty binary relation on $X$ for $0\leq i\leq s$. The \defn{adjacency matrix} $A_i$ of relation $R_i$ is defined to be the $(0,1)$-matrix whose rows and columns are indexed by $X$ such that $(A_i)_{xy}=1$ if $(x,y)\in R_i$ and $(A_i)_{xy}=0$ otherwise. 
A pair $(X,\{R_i\}_{i=0}^s)$ is a \defn{commutative association scheme}, or simply a \defn{scheme} if the following five conditions hold:

\begin{enumerate}[(i)]
\item $A_0$ is the identity matrix.
\item $\sum_{i=0}^sA_i=J$, where $J$ is the all-one matrix.
\item $A_i^T=A_{i'}$ for some $i' \in\{0,1,\ldots,s\}$.
\item $A_iA_j=\sum_{k=0}^sp_{i,j}^kA_k$ for $i,j\in\{0,1,\ldots,s\}$.
\item $A_iA_j=A_jA_i$ for any $i,j$.
\end{enumerate}
(See \cite{BI}, \cite{bcn} for background.) For simplicity, we also refer to the set $\{A_0,\ldots,A_s\}$ as an association scheme. The scheme is said to be \defn{symmetric} if $i'=i$ for all $1\leq i\leq s$, otherwise it is said to be \defn{nonsymmetric}. The algebra $\mathcal{A}$ generated by all adjacency matrices $A_0,A_1,\ldots,A_s$ over $\cx$ is called the \defn{adjacency algebra} or \defn{Bose-Mesner algebra}. If $\mathcal{A}$ is the space spanned by $(0,1)$-matrices $A_0,A_1,\ldots,A_s$ satisfying (i) to (iii), then $\cA$ is the adjacency algebra of an association scheme if and only if 
$\cA$ is commutative and closed under ordinary multiplication.


Since the adjacency algebra is semisimple and commutative, 
there exists a unique set of primitive idempotents of the adjacency algebra, which is denoted by $\{E_0,E_1,\ldots,E_s\}$. 
Since $\{E_0^T,E_1^T,\ldots,E_s^T\}$ forms also the set of primitive idempotents, 
we define $\widehat{i}$ by the index such that $E_{\widehat{i}}=E_i^T$ for $0\leq i\leq s$.
The adjacency algebra is closed under the entrywise product $\circ$, so 
we can define structure constants, the \defn{Krein parameters} $q_{i,j}^k$, for $E_0,E_1,\ldots,E_s$ under entrywise product:
$$E_i\circ E_j=\frac{1}{\abs{X}}\sum_{k=0}^sq_{i,j}^kE_k.$$
Both sets of matrices $\{A_0,A_1,\ldots,A_s\}$ and $\{E_0,E_1,\ldots,E_s\}$ are bases for the adjacency algebra. Therefore there exist change of basis matrices $P$ and $Q$ defined as follows;
\[
A_i=\sum_{j=0}^sP_{ji}E_j,\quad E_j=\frac{1}{\abs{X}}\sum_{i=0}^sQ_{ij}A_i.
\]
We call $P$ and $Q$ the \defn{eigenmatrix} and \defn{second eigenmatrix} of the scheme respectively. 
For each $0\leq i \leq s$, $k_i:=P_{i0}$ and $m_i:=Q_{i0}$ are called the $i$-th valency and multiplicity.

From now on, consider a finite set $X$ in $\Om(d)$ with an inner product set $A(X)=\{\alpha_1,\ldots,\alpha_s\}$, and set $\alpha_0=1$.
For $0\leq i\leq s$, define the relation $R_i$ as the set of pairs $(x,y)$ such that $x^*y=\alpha_i$, and $A_i$ coincides with the adjacency matrix of $R_i$. Then $\{A_0,A_1,\ldots,A_s\}$ clearly satisfy the above conditions from (i) to (iii).
Define the intersection numbers for $x,y\in X$, $1\leq i,j\leq s$ as 
$$p_{i,j}(x,y):=|\{z\in X: x^*z=\alpha_i,z^*y=\alpha_j\}|.$$
For $0 \leq i \leq s$, we let $\tilde{i}$ denote the index such that $\alpha_{\tilde{i}} = \overline{\alpha_i}$.
If the intersection numbers $p_{i,j}(x,y)$ depend only on $i,j$ and $x^*y$ (not on the particular choice of $x$ and $y$), and $p_{i,j}(x,y)=p_{j,i}(x,y)$ holds for all $i,j$, 
then the set $X$ carries an association scheme. 

From each $(k,l)\in\nn^2$ and characteristic matrix $H_{k,l}$, 
define a matrix $F_{k,l}=\frac{1}{|X|}H_{k,l}H_{k,l}^*$.
It follows from the Addition theorem (Theorem~\ref{thm:addition}) that $F_{k,l}=\frac{1}{|X|}\sum_{i=0}^sg_{k,l}(\alpha_i)A_i$, 
which means that each $F_{k,l}$ is in the vector space $\mathcal{A}$.
When $X$ is a design, $F_{k,l}$ often appears as a primitive idempotent in the scheme.
The following theorem is a complex analogue of Theorem~7.4 in \cite{dgs}.

\begin{theorem}\label{thm:S1}
Let $X$ be a $\cU*\cU$-design with degree $s$. Then:
\begin{enumerate}[(i)]
\item $|\mathcal{U}|\leq s+1$.
\item If $|\mathcal{U}| \geq s$, then $X$ carries an association scheme.
\item If $|\mathcal{U}|=s+1$, then $X$ is a tight design with respect to $\cU$.
\end{enumerate} 
\end{theorem}

\begin{proof}
\begin{itemize}
\item[(i)]The vector space $\mathcal{A} := \spn\{A_0,\ldots,A_d\}$ has dimension $s+1$.
Since $X$ is a $\cU*\cU$-design, 
the set $\{F_{k,l}: (k,l)\in\mathcal{U}\}$ is linearly independent in $\mathcal{A}$ by Lemma~\ref{lem:designchar}. 
Therefore $|\mathcal{U}|\leq s+1$. 
\item[(ii)]
Set $\cE=\{F_{k,l}:(k,l) \in \cU\}$ if $|\cU| = s+1$ and $\cE = \{F_{k,l}:(k,l) \in \cU\} \cup \{ I-\sum_{(k,l)\in\mathcal{U}}F_{k,l}\}$ if $|\cU| = s$.
Since $|\cE| = s+1$, $\cE$ forms a basis for $\mathcal{A}$ consisting of mutually orthogonal idempotents. Therefore $\mathcal{A}$ is commutative and closed under ordinary multiplication, so it is the adjacency algebra of an association scheme.
\item[(iii)]
For $(k,l)\in \mathcal{U}$, the multiplicity of $F_{k,l}$ in the association scheme is $m_{k,l}$. When $|\cU|=s+1$, it follows that $|X|=\sum_{(k,l)\in\mathcal{U}}m_{k,l}$, attaining the bound in Theorem~\ref{thm:absbound} (i).
Hence $X$ is a tight design with respect to $\cU$.
\end{itemize}
\end{proof}
When the assumption of Theorem~\ref{thm:S1}(ii) holds, the set $\cE$ is precisely the set of primitive idempotents of the association scheme.
Moreover, every idempotent of $\cE \backslash \{I-\sum_{(k,l)\in\mathcal{U}}F_{k,l}\}$ has the form $g_{k,l}(F_{1,0})$ or $g_{k,l}(F_{0,1})$, where $F_{1,0}$ is the primitive idempotent which is a multiple of the Gram matrix of $X$.

In case of $|\cU|=s+1$, the second eigenmatrix is given by 
\begin{equation}\label{eqn:secondeigenmatrix}
Q=(g_{k,l}(\alpha_i))_{\substack{1\leq i\leq s\\ (k,l)\in\mathcal{U}}}.
\end{equation}
The following theorem constrains the inner product set $A(X)$ in this case. 

\begin{corollary}\label{S12}
Let $X$ be a $\cU*\cU$-design with degree $s$ such that $|\cU|=s+1$.
Then each element $\alpha_i \in A(X)$ is a root of $\sum_{(k,l)\in\mathcal{U}}g_{k,l}(x)$. 
\end{corollary}

\begin{proof}
For distinct $x,y$ such that $x^*y=\alpha_i$, comparing the $(x,y)$ entry of both sides of the equation $\sum_{(k,l)\in\mathcal{U}}H_{k,l}H_{k,l}^*=|X|I$, 
we have $\sum_{(k,l)\in\mathcal{U}}g_{k,l}(\alpha_i)=0$. 
\end{proof}

Note, however, that not every tight $\cU$-code has degree $s = \abs{\cU}-1$, as the following example indicates.

\begin{example}
Consider $\cU= \{(0,0),(1,0),(0,1)\}$. Let $X$ be a subset of $\Om(d)$ such that $\phi(X)$ is a regular simplex in $\re^{2d}$ (where $\phi$ is the natural embedding of $\cx^d \rightarrow \re^{2d}$ given in equation~\eqref{eqn:phi}). Then $|X| = 2d+1$ and for all $x \neq y$ in $X$,
\[
\Re(x^*y) = \phi(x)^T\phi(y) = -1/2d.
\]
Thus $\sum_{(k,l) \in \cU}g_{k,l}(x) = d(x + \bar{x}) + 1$ is an annihilator for $X$, and $X$ is a tight $\cU$-code. However, the angle set $A(X)$ may be large, as it is only $\Re(x^*y)$ that is constrained, not $x^*y$ itself. In general, $X$ has degree larger than $\abs{\cU} - 1 = 2$.
Indeed, if degree is $2$, then $X$ carries a nonsymmetric scheme with class $2$.
Then a necessary condition of existence for such schemes is that $d$ is congruent to $3$ modulo $4$.
\end{example}

\section{Association schemes related to real and projective designs}\label{sec:connections}

Lemmas~\ref{lem:projectivedesign} and \ref{lem:realdesign} showed that complex spherical designs can sometimes be constructed from projective or real spherical designs and vice versa. The following theorems show that the corresponding association schemes are also related.

A scheme $(X,\{\tilde{R}_i\})$ is a \defn{fusion} of scheme $(X,\{R_i\})$ if each $\tilde{R}_i$ is a union of $R_i$'s.

\begin{theorem}\label{thm:complexreal}
Let $t$ be a positive even integer and 
let $X$ be a tight design with respect to $\mathcal{U}=\{(k,l)\in\nn^2:  k+l\leq t\}$ and $|\cU|\geq s$.
Then:
\begin{enumerate}[(i)]
\item $\phi(X)$ is a tight $t$-design in $S^{2d-1}$. 
\item The scheme $(\phi(X),\{R_\alpha: \alpha\in A(\phi(X))\}$ is a fusion scheme of the scheme $(X,\{R_\alpha: \alpha \in A(X)\})$.
\end{enumerate}
\end{theorem}

\begin{proof}
\begin{enumerate}[(i)]
\item Lemma~\ref{lem:realdesign} implies that $\phi(X)$ is a $t$-design. Since
\[
|\phi(X)|=|X|=\sum_{(k,l)\in \mathcal{U}}m_{d,k,l}=\binom{2d+t/2-1}{t/2}+\binom{2d+t/2-2}{t/2-1},
\]
the lower bound on the size of a $t$-design in equation~\eqref{eqn:realbound} implies that $\phi(X)$ is tight.

\item Since $\phi(X)$ is a tight $t$-design, the set $A(\phi(X))=\{\text{Re}(\alpha): \alpha \in A(X))\}$ coincides with the entire set of roots of Gegenbauer polynomial $Q_{2d,t/2}(x)$.
In particular the cardinality of $A(\phi(X))$ is $t/2$.
Here, we consider partitions of adjacency matrices and primitive idempotents as follows:
\begin{align*}
\{A_\alpha: \alpha\in A(X)\}&=\bigcup_{\alpha'\in A(\phi(X))}\{A_\alpha: \text{Re}(\alpha)=\text{Re}(\alpha')\},\\
\{E_{k,l}: (k,l)\in \mathcal{U}\}&=\bigcup_{0\leq n\leq t}\{E_{k,l}: k+l=n\}.
\end{align*}
Consider a block in the second eigenmatrix $Q$ with rows indexed by $\alpha'\in A(\phi(X))$ and columns indexed by $n \leq t$. By equation~\eqref{eqn:poly}, the row sum of that block is $ Q_{2d,n}(\alpha')$ . Then the  Bannai-Muzychuk criterion \cite{bannai,muz} shows that these partitions give a fusion scheme, 
and its second eigenmatrix is 
\[
(Q_{2d,n}(\alpha'))_{\substack{\alpha'\in A(\phi(X))\\ 0\leq n\leq t/2}}.
\]
Hence this fusion scheme coincides with the scheme obtained from the tight spherical $t$-design $\phi(X)$.
\end{enumerate}
\end{proof}

A scheme $(\tilde{X},\{\tilde{R}_i\})$ is a \defn{quotient} of scheme $(X,\{R_i\})$ if some union of $R_i's$ is an equivalence relation on $X$, with equivalence classes $\tilde{X}$, and $\{\tilde{R}_i\}$ is the set of relations induced from $\{R_i\}$ by that equivalence relation (see \cite[Section 2.4]{bcn}).

\begin{theorem}\label{thm:complexprojective}
Let $X$ be an $n$-antipodal cover of $L$ of degree $s$ such that for every $\al \in \{1\} \cup A(L)\setminus\{0\}$, there are exactly $n$ elements of  $\{1\} \cup A(X)\setminus\{0\}$ with absolute value $\al$. Suppose there exists a lower set $\mathcal{U}$ such that $\cU*\cU\sbs\cT$ and $s\leq |\mathcal{U}|$, and let $t$ be the largest integer with $(t,t) \in \cT$, with $2 \leq t \leq n$.
Then:
\begin{enumerate}[(i)]
\item $P(L)$ is a $t$-design in $\mathbb{C}\mathbb{P}^{d-1}$ with degree at most $t/2+1$.
\item The scheme $(P(L),\{R_\alpha: \alpha\in A(P(L))\}$ is a quotient scheme of the scheme $(X,\{R_\alpha: \alpha \in A(X)\})$.
\end{enumerate}
\end{theorem}

\begin{proof}
\begin{enumerate}[(i)]
\item  By Lemma~\ref{lem:projectivedesign}, $P(L)$ is a $t$-design. Since $(t+1,t+1)\not\in\mathcal{T}$ we know $\mathcal{U}$ is contained in $\{(k,l): k+l\leq t\}$, and so $|\mathcal{U}|\leq \frac{(t+1)(t+2)}{2}$. Let $r$ denote the degree of $P(L)$. 

Consider first the case when $0\not\in A(X)$: then comparing $|A(L)|$ and $|A(X)|$ we have $n(r+1) = s+1$. Since $t \leq n$, 
\begin{align*}
n(r+1) = s+1\leq |\mathcal{U}|+1\leq \frac{(t+1)(t+2)}{2}+1\leq n(\frac{t}{2}+2).
\end{align*}
Similarly, when $0\in A(X)$,
\begin{align*}
nr+1 = s+1\leq |\mathcal{U}|+1\leq \frac{(t+1)(t+2)}{2}+1\leq n\frac{t+2}{2}+1.
\end{align*}
In either case, $r\leq t/2+1$.
\item Set $R_i=\{(x,y)\in X:x^*y=w_n^i\}$.
The set $\cup_{i=0}^{n-1}R_i$ is an equivalence relation, and 
let $\Sigma$ be the system of imprimitivity.
Then $\Sigma$ coincides with $P(L)$.

Define an equivalence relation on $\{0,1,\ldots,s\}$ as follows: $i$ and  $j$ are equivalent if and only if $p_{i,k}^j\neq0$ for some $k$ such that $\alpha_k$ is a multiple of $w_n$. It follows that $i$ and $j$ are equivalent if and only if $\al_i$ and $\al_j$ have the same absolute value. Therefore the scheme derived from the projective design $P(L)$ coincides with the quotient scheme of the scheme derived from complex spherical design $X$.     
\end{enumerate}
\end{proof}

\begin{example}\label{27}\cite{cox}
Let $\om$ be a primitive third root of unity and let $X \sbs \Omega(3)$ be a set of vectors $\frac{1}{\sqrt{2}}(0, w^\mu,- w^\nu)$, $\frac{1}{\sqrt{2}}(- w^\mu,0, w^\nu)$, $\frac{1}{\sqrt{2}}(w^\mu,- w^\nu,0)$ for $\mu,\nu\in\{0,1,2\}$.
Then $|X| = 27$, $A(X)=\{w^j,-\frac{1}{\sqrt{2}}w^j: 0\leq j\leq2\}$ with degree $s=5$, and $X$ is $\mathcal{T}$-design where
$$\cT=\cl(\{(5,0),(3,2),(2,3),(0,5)\}).$$
We can take $\mathcal{U}=\{(i,j)\in\nn^2: i+j\leq2\}$. Then $|\mathcal{U}|=s+1$, and by Theorem~\ref{thm:S1}, $X$ carries a nonsymmetric association scheme.
Note that $X$ satisfies the absolute bound in Theorem~\ref{thm:absbound} (i), and
Theorem~\ref{thm:complexreal} applies so we obtain tight real $4$-design in $S^{6}$.
Theorem~\ref{thm:complexprojective} also applies, so we obtain tight projective $2$-design in $\mathbb{C}\mathbb{P}^{2}$.

\end{example}
\begin{example}{\bf(MUBs in $\cx^2$)} 
Let $L$ be a complex set of MUBs in $\mathbb{C}^2$ with the inner product set  $A(L)=\{\frac{\pm1\pm\mathrm{i}}{2},0\}$. For example,
\[
L = \{(1,0),(0,1)\} \cup \{\tfrac{1+\mathrm{i}}{2}(1,\mathrm{i}^j): 0 \leq j \leq 3\}.
\]
Define $X$ to be a $4$-antipodal cover of $L$, so $A(X)=\{-1,\pm \mathrm{i},\frac{\pm1\pm\mathrm{i}}{2},0\}$.
Then $X$ is a $\mathcal{T}$-design where
$\cT=\cl(\{(7,0),(4,3),(3,4),(0,7)\})$, and we can take $\mathcal{U}=\cl(\{(3,0),(1,1),(0,3)\})$ such that $\mathcal{U}*\mathcal{U}\sbs\cT$ and $|\cU|=s$.
From Theorem~\ref{thm:S1} it follows Theorem $X$ carries a nonsymmetric association scheme.
\end{example}

\section{Association schemes from antipodal designs}\label{sec:antipodal}

Even if the assumptions of Theorem~\ref{thm:S1} do not hold, the inner product relations of a finite set $X \sbs \Om(d)$ might still carry an association scheme. For example, suppose $X$ is an $n$-antipodal set and  the first $n$ angles of $A(X)$ are $\alpha_0 = 1, \alpha_1 = \om_n,\ldots, \alpha_{n-1} = \om_n^{n-1}$. If either $i$ or $j$ is less than $n$, then the intersection number $p_{i,j}(x,y)$ depends only on $i,j$, and $x^*y$. Moreover, $p_{i,j}(x,y)=p_{j,i}(x,y)$. In such situations $X$ often gives rise to an association scheme, as Theorem~\ref{thm:S8} below shows.

We also show a sufficient condition for a finite set $X \sbs \Om(d)$ to satisfy another regularity property: a subset $X \sbs \Omega(d)$ with the inner product set $A(X)=\{\alpha_1\ldots,\alpha_s\}$ is called \defn{inner product invariant} if $k_i(x):=|\{y\in X: x^*y=\alpha_i\}|$ does not depend on the choice of $x\in X$ for each $1\leq i\leq s$.
The value $k_i(x)=|\{y\in X: x^*y=\alpha_i\}|$ is called the $i$-th valency of $x\in X$.
When $X$ is inner product invariant, $k_i(x)$ is abbreviated as $k_i$. 
It is clear that $X$ is inner product invariant if and only if the all ones vector is an eigenvector of $A_i$ for each $1\leq i\leq s$.

In order to establish that a $\cT$-design is inner product invariant or carries an association scheme, we introduce a matrix which we call the \defn{Jacobi matrix}. The Jacobi matrix $G$ has rows indexed by a set of inner products $A$ and columns indexed by a set $\cU \sbs \nn^2$ such that $\abs{A} = \abs{\cU}$, and 
\[
G=(g_{k,l}(\alpha))_{\substack{\al \in A\\ (k,l)\in\mathcal{U}}}.
\]
We require the matrix to be nonsingular. Note that if $\cU$ is a lower set, $G$ can be obtained from the matrix of monomials $M = (\al^k\overline{\al}^l)_{\substack{\al \in A\\ (k,l)\in\mathcal{U}}}$ by elementary row and column operations. If $\cU = \{(0,0),(1,0),\ldots,(t,0)\}$, then $G$ is nonsingular because it can be obtained by elementary operations from a Vandermonde matrix. But in general it can be singular. For example, if $\cU$ is contained in $\{(k,k): k \in \nn\}$ and $A$ contains both $\al$ and $\overline{\al}$ for some $\al \notin \re$, then $G$ contains two identical rows and so $G$ is singular. 

\begin{theorem}\label{thm:S8}
Let $X$ have inner product set $A(X)=\{\alpha_1,\ldots,\alpha_s\}$.
\begin{enumerate}[(i)]
\item Suppose that $X$ is a $\cU$-design with $|\mathcal{U}|=s+1$ and set $\al_0 = 1$. If the matrix $G=(g_{k,l}(\alpha_i))_{\substack{0\leq i\leq s\\ (k,l)\in\mathcal{U}}}$ is nonsingular, then $X$ is inner product invariant.
\item Suppose that $X$ is a $\mathcal{U}*\mathcal{U}$-design and there is some index set $I\sbs \{1,2,\dots,s\}$ with $|I| = |\mathcal{U}|$ such that if either $i$ or $j$ is not in $I$, then the intersection numbers $p_{i,j}(x,y)$ satisfy $p_{i,j}(x,y)=p_{j,i}(x,y)$ and depend only on $i$, $j$, and $x^*y$. If the matrix $G=(g_{k,l}(\alpha_i))_{\substack{i\in I\\ (k,l)\in\mathcal{U}}}$ is nonsingular, then $X$ carries an association scheme.
\end{enumerate}
\end{theorem}

\begin{proof}
\begin{enumerate}
\item[(1)] Define a vector space $\mathcal{A}=\spn\{A_0,A_1,\ldots,A_s\}$.
From Theorem~\ref{thm:addition}, it follows that $F_{k,l}=\frac{1}{|X|}\sum_{i=0}^sg_{k,l}(\alpha_i)A_i$ for each $(k,l)\in\nn^2$.
Since $\{A_0,A_1,\ldots,A_s\}$ is a basis of $\mathcal{A}$ and the assumption that $G$ is nonsingular, 
$\{F_{k,l}:(k,l)\in \cU\}$ is also a basis of $\mathcal{A}$.

Since $X$ is a $\cU$-design, Lemma~\ref{lem:designchar} shows 
$H_{k,l}^*H_{0,0}=|X|\delta_{k,0}\delta_{l,0}I$ for each $(k,l)\in\cU$.
Premultipling $H_{k,l}$ yields $F_{k,l}H_{0,0}=|X|\delta_{k,0}\delta_{l,0}H_{0,0}$.
The matrix $H_{0,0}$ is the all ones column vector, so this implies the all ones vector is an eigenvector of $F_{k,l}$.
Therefore all ones vectors is also an eigenvector of $A_i$ for each $0\leq i\leq s$, so $X$ is inner product invariant.

\item[(2)] It suffices to prove that for any $x,y\in X$ and  $i,j\in I$, the intersection number $p_{i,j}(x,y)$ depends only on $i,j$ and $x^*y$,  and satisfies $p_{i,j}(x,y)=p_{j,i}(x,y)$.
Fix $x,y\in X$. For $(k,l),(k',l')\in\mathcal{U}$, it follows that 
\begin{align}
F_{k,l}F_{k',l'}=\delta_{k,k'}\delta_{l,l'}|X|F_{k,l}.
\end{align}
Using $F_{k,l}=\sum_{i=0}^s g_{k,l}(\alpha_i)A_i$, we have 
\begin{align}\label{S2}
\Big(\sum_{i=0}^s g_{k,l}(\alpha_i)A_i\Big)\Big(\sum_{j=0}^s g_{k',l'}(\alpha_j)A_j\Big)=\delta_{k,k'}\delta_{l,l'}|X|\sum_{i=0}^s g_{k,l}(\alpha_i)A_i.
\end{align}

Let $L = \{1,\ldots,s\}\times\{1,\ldots,s\} \setminus I\times I$. Then the $(x,y)$-entry of LHS of equation (\ref{S2}) is 
\begin{align}\label{S3}
\lefteqn{\Big(\Big(\sum_{i=0}^s g_{k,l}(\alpha_i)A_i\Big)\Big(\sum_{j=0}^s g_{k',l'}(\alpha_j)A_j\Big)\Big)_{xy}}\notag\\
& \quad =\sum_{i,j=0}^s g_{k,l}(\alpha_i)g_{k',l'}(\alpha_j)p_{i,j}(x,y)\notag \displaybreak[0]\\
& \quad =\sum_{i,j\in I} g_{k,l}(\alpha_i)g_{k',l'}(\alpha_j)p_{i,j}(x,y)+p_{0,0}(x,y)+p_{0,h}(x,y)+p_{\tilde{h},0}(x,y)\notag \displaybreak[0]\\
&\qquad +\sum_{(i,j)\in L} g_{k,l}(\alpha_i)g_{k',l'}(\alpha_j)p_{i,j}(x,y),\displaybreak[0]
\end{align}
where $h$ is the index such that $x^*y = \alpha_h$. The $(x,y)$-entry of RHS of equation (\ref{S2}) is
\begin{align}\label{S4}
\delta_{k,k'}\delta_{l,l'}|X| g_{k,l}(\alpha_h).
\end{align}
Substituting (\ref{S3}) and (\ref{S4}) into (\ref{S2}), we obtain
\begin{align}\label{S5}
\sum_{i,j\in I} g_{k,l}(\alpha_i)g_{k',l'}(\alpha_j)p_{i,j}(x,y)=\delta_{k,k'}\delta_{l,l'}|X| g_{k,l}(\alpha_h)-p_{0,0}(x,y)\notag \displaybreak[0]&\\
-p_{0,h}(x,y)-p_{\tilde{h},0}(x,y) -\sum_{(i,j)\in L} g_{k,l}(\alpha_i)g_{k',l'}(\alpha_j)p_{i,j}(x,y)&.
\end{align}
The assumption (ii) implies that the RHS of equation (\ref{S5}) depends only on $\alpha_h$.
For $(k,l),(k',l')\in\mathcal{U}$, equation (\ref{S5}) yields a system of linear equations whose unknowns are $\{p_{i,j}(x,y): i,j\in I\}$.
Its coefficient matrix $G\otimes G$ is nonsingular from the assumption (iii).
Therefore $p_{i,j}(x,y)$ for $i,j\in I$ depends only on $\alpha_h$ and does not depend on the choice of $x$, $y$ satisfying $\alpha_h=x^*y$.

By replacing $(k,l),(k',l'),i,j$ with  $(k',l'),(k,l),j,i$ in equation (\ref{S5}), and recalling that  $p_{i,j}(x,y)=p_{j,i}(x,y)$ for $i,j\in L$, we see that the system of equations for $\{p_{i,j}(x,y): i,j \in I\}$ is symmetric in $i$ and $j$. Therefore $p_{i,j}(x,y)=p_{j,i}(x,y)$ holds for any $i,j \in I$, and hence $X$ carries an association scheme.
\end{enumerate}
\end{proof}

The following give examples that have the property of inner product invariant, do not carry association schemes.
\begin{example}\cite{cox}
Let $\lambda=(-1-\sqrt{7}\mathrm{i})/2$, and let $X$ be a set of permutations of vectors of the form
\begin{align*}
&\tfrac{\pm1}{2\sqrt{2}}( \lambda^2, \lambda^2,0),& &\tfrac{\pm 1}{2\sqrt{2}}(\lambda+2,-\lambda-2,0),& &\tfrac{\pm 1}{\sqrt{2}}(\lambda,0,0), \\
&\tfrac{\pm 1}{2\sqrt{2}}(\lambda,\lambda,2),& &\tfrac{\pm 1}{2\sqrt{2}}(\lambda,-\lambda,2),&  &
\tfrac{\pm 1}{2\sqrt{2}}(-\lambda,-\lambda,2).
\end{align*}
Then $|X|=42$. $A(X)=\{-1,\pm\tfrac{1}{2},0,\tfrac{\pm 1\pm \sqrt{7}\mathrm{i}}{4}\}$ with degree $s=8$, and $X$ is $\mathcal{T}$-design where $\mathcal{T}=\cl(\{(3,2),(2,3)\})$.
We can take $\cS=\cl(\{(3,2),(0,2)\})$, 
so the determinant of 
$G=(g_{k,l}(\alpha))_{\substack{\al\in A(X)\\ (k,l)\in\cS}}$ is $-3087/2048$ and therefore $G$ is nonsingular.
Hence by Proposition~\ref{thm:S8}(i), $X$ is inner product invariant.
However, $X$ does not carry an association scheme.
Indeed, if we set 
\[
(\al_i)_{i=0}^8=(1,-1,0,\tfrac{1}{2},-\tfrac{1}{2},\tfrac{ 1+ \sqrt{7}\mathrm{i}}{4},\tfrac{-1-\sqrt{7}\mathrm{i}}{4},\tfrac{1-\sqrt{7}\mathrm{i}}{4},\tfrac{-1+ \sqrt{7}\mathrm{i}}{4})
\]
then $A_iA_j\in \cA=\text{Span}(A_0,\ldots,A_8)$ only for $i=0,1$, $j=0,1$, or $(i,j)\in\{2,3,4,5\}^2\cup\{6,7\}^2\cup\{8,9\}^2$.

On the other hand, $\phi(X)$ with its inner products carries a symmetric association scheme.
\end{example}

\begin{example}\cite{cox}
Let $\lambda=(-1-\sqrt{7}\mathrm{i})/2$, and let $X$ be a set of permutation of vectors of the form
\begin{align*}
&\tfrac{1}{\sqrt{6}}(\pm \lambda,\pm \lambda,\pm \lambda),& &\tfrac{1}{\sqrt{6}}(\pm \lambda^2,\pm 1,\pm  1),& &\tfrac{1}{\sqrt{6}}(\pm \co{\lambda}^2,\pm \co{\lambda},0).
\end{align*}
Then $|X|=56$, $A(X)=\{-1,\pm\tfrac{1}{2},0,\tfrac{\pm 1\pm \sqrt{7}\mathrm{i}}{4}\}$ with degree $s=8$, and $X$ is $\mathcal{T}$-design where $\mathcal{T}=\cl(\{(3,2),(2,3)\})$.
We can take $\cS=\cl(\{(3,2)\})$, 
so the determinant of 
$G=(g_{k,l}(\alpha))_{\substack{\al\in A(X)\\ (k,l)\in\cS}}$ is $-11014635520000000000\sqrt{7}\mathrm{i}/43046721$ and therefore $G$ is nonsingular.
Hence by Proposition~\ref{thm:S8}(i), $X$ is inner product invariant.
However, $X$ does not carry an association scheme.
Indeed, if we set 
\[
(\al_i)_{i=0}^{11}=(1,-1,\tfrac{1}{3},-\tfrac{1}{3},\tfrac{2}{3},-\tfrac{2}{3},\tfrac{\sqrt{7}\mathrm{i}}{3},-\tfrac{\sqrt{7}\mathrm{i}}{3},\tfrac{1+\sqrt{7}\mathrm{i}}{6},\tfrac{-1-\sqrt{7}\mathrm{i}}{6},\tfrac{1-\sqrt{7}\mathrm{i}}{6},\tfrac{-1+\sqrt{7}\mathrm{i}}{6})
\]
then $A_iA_j\in \cA=\text{Span}(A_0,\ldots,A_{11})$ only for $i=0,1$, $j=0,1$, or $(i,j)\in\{2,3,4,5\}^2\cup\{6,7\}^2\cup\{8,9\}^2\cup\{10,11\}^2\cup\{4,5\}\times\{6,7\}\cup\{6,7\}\times\{4,5\}$.

On the other hand, $\phi(X)$ with its inner products carry a symmetric association scheme.
\end{example}

The following give examples fitting into Theorem~\ref{thm:S8}(ii).
\begin{example}\cite{cox}
Let $\om$ be a $6$-th primitive root of unity, and let $X$ be a set of vectors of the form
\begin{align*}
&\tfrac{1}{\sqrt{3}}(0,\pm w^\mu,\pm w^\nu,\pm w^\lambda), (\pm\textrm{i}w^\lambda,0,0,0), \tfrac{1}{\sqrt{3}}(\mp w^\mu,0,\pm w^\nu,\pm w^\lambda),(0,\pm\textrm{i}w^\lambda,0,0), \\
&\tfrac{1}{\sqrt{3}}(\pm w^\mu,\mp w^\nu,0,\pm w^\lambda),(0,0,\pm\textrm{i}w^\lambda,0), 
\tfrac{1}{\sqrt{3}}(\mp w^\mu,\mp w^\nu,\mp w^\lambda,0), 
(0,0,0,\pm\textrm{i}w^\lambda)
\end{align*}
for $\lambda,\mu,\nu\in\{0,1,2\}$.
Then $|X|=240$ and $A(X)=\{0,w^j,\tfrac{\textrm{i}}{\sqrt{3}}w^j: 0\leq j\leq5\} \setminus \{1\}$ with degree $s=12$, and $X$ is $\mathcal{T}$-design where $\mathcal{T}=\{(i,j)\in\nn^2: i+j\leq7\}$.
Set
\[
(\al_i)_{i=0}^{12} = (1,\om,\ldots,\om^5,0,\tfrac{\textrm{i}}{\sqrt{3}},\tfrac{\textrm{i}}{\sqrt{3}}\om,\ldots,\tfrac{\textrm{i}}{\sqrt{3}}\om^5).
\]
We can take $\mathcal{U}=\{(i,j)\in\nn^2: i+j\leq2\}\cup \{(3,0)\}$, so the determinant of 
$G=(g_{k,l}(\alpha_i))_{\substack{6\leq i\leq 12\\ (k,l)\in\mathcal{U}}}$ is $8\mathrm{i}/9\sqrt{3}$ and therefore $G$ is nonsingular.
Hence by Theorem~\ref{thm:S8}, $X$ carries a nonsymmetric association scheme.
\end{example}

\begin{example}\label{756}\cite{cox}
Let $\om$ be a $6$-th primitive root of unity, and let $X$ be a set of permutations of vectors 
\begin{align*}
& ((-1)^{k_1},(-1)^{k_2},0,0,0,0) \mbox{ and } \\
& ((-1)^{k_1}\sqrt{3}\textrm{i},(-1)^{k_2},(-1)^{k_3},(-1)^{k_4},(-1)^{k_5},(-1)^{k_6}),
\end{align*}
where $k_1,\ldots,k_6\in\{0,1\}$ and $k_1+\cdots+k_6$ is even.
Then $|X| = 756$, $A(X)=\{0,w^j,\frac{2}{3}w^j: 0\leq j\leq5\} \setminus \{1\}$ with degree $s=12$, and $X$ is $\mathcal{T}$-design where $\cT=\cl(\{(5,3),(3,5)\})$.
Set
\[
(\al_i)_{i=0}^{12} = (1,\om,\ldots,\om^5,0,\tfrac{2}{3},\tfrac{2}{3}\om,\ldots,\tfrac{2}{3}\om^5).
\]
We take $\mathcal{U}=\{(i,j)\in\nn^2: i+j\leq2\}\cup \{(3,0)\}$, 
so the determinant of 
$G=(g_{k,l}(\alpha_i))_{\substack{6\leq i\leq 12\\ (k,l)\in\mathcal{U}}}$ is $-27/256$ and again by Theorem~\ref{thm:S8} $X$ carries a nonsymmetric association scheme.
\end{example}

\begin{example}({\bf Complex MUBs in $\cx^{2^m}$, $m$ even})\label{complexMUBeven}
Let $L$ be a set of complex MUBs in $\mathbb{C}^d$ with $|L|=d(d+1)$ and $A(L)=\{\pm\frac{1}{\sqrt{d}},\pm\frac{\mathrm{i}}{\sqrt{d}},0\}$. (Such MUBs can be constructed from $\zz_4$-Galois rings \cite{kr}).
Define $X$ to be a $4$-antipodal cover of $L$, so $A(X)=\{\pm\mathrm{i},-1,\pm\frac{1}{\sqrt{d}},\pm\frac{\mathrm{i}}{\sqrt{d}},0\}$.
Set 
$$(\alpha_i)_{i=0}^8=\Big(1,\mathrm{i},-1,-\mathrm{i},\tfrac{1}{\sqrt{d}},\tfrac{\mathrm{i}}{\sqrt{d}},\tfrac{-1}{\sqrt{d}},\tfrac{-\mathrm{i}}{\sqrt{d}},0\Big).$$ 

Then $X$ is inner product invariant with valencies 
$$(k_i)_{i=0}^8=(1,1,1,1,d^2,d^2,d^2,d^2,4(d-1)),$$
and $X$ is a $\mathcal{T}$-design with $\cT=\cl(\{(3,2),(2,3)\}).$
If we take $\mathcal{U}=\cl(\{(1,1),(2,1)\})$, then $\mathcal{U}*\mathcal{U}\subset\mathcal{T}$ and $|\mathcal{U}|=5$.
For $i$ or $j$ at most $3$, the intersection numbers $p_{i,j}(x,y)$ are determined by $x^*y$, and $p_{i,j}(x,y)=p_{j,i}(x,y)$. From direct calculation the determinant of the matrix $G=(g_{k,l}(\alpha_i))_{\substack{4\leq i\leq 8\\ (k,l)\in\mathcal{U}}}$ is $16\mathrm{i}/d^3$. Therefore $X$ carries a nonsymmetric association scheme by Theorem~\ref{thm:S8}.
\end{example}

\begin{example}({\bf Complex MUBs in $\cx^{2^m}$, $m$ odd})\label{complexMUBodd}
Let $L$ be a set of complex MUBs in $\mathbb{C}^d$ with $|L|=d(d+1)$ and $A(L)=\{\frac{\pm1\pm\mathrm{i}}{\sqrt{2d}},0\}$ (as in \cite{kr}).
Further define $X$ to be a $4$-antipodal cover of $L$. Then $A(X)=\{\frac{\pm1\pm\mathrm{i}}{\sqrt{2d}},0,\pm\mathrm{i},-1\}$.
Set 
$$(\alpha_i)_{i=0}^8=\Big(1,\mathrm{i},-1,-\mathrm{i},\tfrac{1+\mathrm{i}}{\sqrt{2d}},\tfrac{-1+\mathrm{i}}{\sqrt{2d}},\tfrac{-1-\mathrm{i}}{\sqrt{2d}},\tfrac{1-\mathrm{i}}{\sqrt{2d}},0\Big).$$ 

$X$ is inner product invariant with valencies 
$$(k_i)_{i=0}^8=(1,1,1,1,d^2,d^2,d^2,d^2,4(d-1)),$$
and $X$ is a $\mathcal{T}$-design, where $\cT=\cl(\{(7,0),(4,2),(2,4),(0,7)\})$.
We take $\mathcal{U}=\{(k,j): k+j\leq2\}$,
so $\mathcal{U}*\mathcal{U}\subset\mathcal{T}$, $|\mathcal{U}|=6$.
For $i$ or $j$ at most $3$, $p_{i,j}(x,y)$ is determined by $x^*y$ and $p_{i,j}(x,y)=p_{j,i}(x,y)$.
From direct calculation the determinant of the matrix $G=(g_{k,l}(\alpha_i))_{\substack{4\leq i\leq 8\\ (k,l)\in\mathcal{U}}}$ is $-32/d^3$.
Hence $X$ carries a nonsymmetric association scheme by Theorem~\ref{thm:S8}.

In this case, the scheme has a fusion scheme. The second eigenmatrix $Q$ is:
\begin{align}\label{S10}
\begin{pmatrix}
1   &    d   &   d   &   \frac{d(d+1)}{2}   &   \frac{d(d+1)}{2}   &   d^2   &   d^2   &   d^2-1&d\\
1   &    \cxi d   &   -\cxi d   &   \frac{-d(d+1)}{2}   &   \frac{-d(d+1)}{2}   &   -\cxi d^2   &   \cxi d^2   &   d^2-1&d\\
1   &    -\cxi d   &   \cxi d   &   \frac{-d(d+1)}{2}   &   \frac{-d(d+1)}{2}   &   \cxi d^2   &   -\cxi d^2   &   d^2-1&d\\
1   &    -d   &   -d   &   \frac{d(d+1)}{2}   &   \frac{d(d+1)}{2}   &   -d^2   &   -d^2   &   d^2-1&d\\
1   &   \frac{(1+\cxi )\sqrt{d}}{\sqrt{2}}&   \frac{(1-\cxi )\sqrt{d}}{\sqrt{2}}   &  \frac{\cxi (d+1)}{2}    &   \frac{-\cxi (d+1)}{2}   &   \frac{(-1+\cxi )\sqrt{d}}{\sqrt{2}}   &   \frac{(-1-\cxi )\sqrt{d}}{\sqrt{2}}   &   0&-1\\
1   &   \frac{(1-\cxi )\sqrt{d}}{\sqrt{2}}&   \frac{(1+\cxi )\sqrt{d}}{\sqrt{2}}   &  \frac{-\cxi (d+1)}{2}    &   \frac{\cxi (d+1)}{2}   &   \frac{(-1-\cxi )\sqrt{d}}{\sqrt{2}}   &   \frac{(-1+\cxi )\sqrt{d}}{\sqrt{2}}   &   0&-1\\
1   &   \frac{(-1+\cxi )\sqrt{d}}{\sqrt{2}}&   \frac{(-1-\cxi )\sqrt{d}}{\sqrt{2}}   &  \frac{-\cxi (d+1)}{2}    &   \frac{\cxi (d+1)}{2}   &   \frac{(1+\cxi )\sqrt{d}}{\sqrt{2}}   &   \frac{(1-\cxi )\sqrt{d}}{\sqrt{2}}   &  0& -1\\
1   &   \frac{(-1-\cxi )\sqrt{d}}{\sqrt{2}}&   \frac{(-1+\cxi )\sqrt{d}}{\sqrt{2}}   &  \frac{\cxi (d+1)}{2}    &   \frac{-\cxi (d+1)}{2}   &   \frac{(1-\cxi )\sqrt{d}}{\sqrt{2}}   &   \frac{(1+\cxi )\sqrt{d}}{\sqrt{2}}   &  0& -1\\
1&0&0&0&0&0&0&-d-1&d
\end{pmatrix}.
\end{align}
By the Bannai-Muzychuk criterion, the partition of adjacency matrices into sets $\{\{A_0\},\{A_1,A_2,A_8\},\{A_3\},\{A_4,A_5\},\{A_6,A_7\}\}$ and a partition of primitive idempotents $\{\{E_0\},\{E_1,E_2\},\{E_3,E_4,E_7\},\{E_5,E_6\},\{E_8\}\}$ gives a fusion scheme whose second eigenmatrix is 
\begin{align}\label{S11}
\tilde{Q}=
\begin{pmatrix}
1&2d        &(2d-1)(d+1)&2d^2&d\\
1&\sqrt{2d} &0          &-\sqrt{2d}&-1\\ 
1&-\sqrt{2d}&0          &\sqrt{2d} &-1\\
1&0         &-d-1       &0         &d\\
1&-2d       &(2d-1)(d+1)&-2d^2     &d
\end{pmatrix}. 
\end{align}
The fusion scheme coincides with an association scheme obtained from $(d+1)$ real mutually unbiased bases in $\mathbb{R}^{2d}$ \cite{LMO}.
\end{example}

\begin{example}({\bf SIC-POVMs in $\cx^2, \cx^8$})
Let $L$ be a SIC-POVM in $\Omega(d)$ with the inner product set $A(L)=\{\frac{\pm 1}{\sqrt{d+1}},\frac{\mathrm{\pm i}}{\sqrt{d+1}}\}$ (as in Example~\ref{ex:sicpovm}).
Define $X$ to be a $4$-antipodal cover of $L$, so $A(X)=\{\pm\frac{1}{\sqrt{d+1}},\pm\frac{\mathrm{i}}{\sqrt{d+1}},\pm\mathrm{i},-1\}$.
Set 
$$(\alpha_i)_{i=0}^7=\Big(1,\mathrm{i},-1,-\mathrm{i},\tfrac{1}{\sqrt{d+1}},\tfrac{\mathrm{i}}{\sqrt{d+1}},\tfrac{-1}{\sqrt{d+1}},\tfrac{-\mathrm{i}}{\sqrt{d+1}},\mathrm{i},-1,-\mathrm{i}\Big).$$ 
Then $X$ is inner product invariant with valencies 
$$(k_i)_{i=0}^7=(1,1,1,1,d^2-1,d^2-1,d^2-1,d^2-1),$$
and $X$ is a $\mathcal{T}$-design with $\cT=\cl(\{(3,2),(2,3)\})$.
If $\mathcal{U}=\cl(\{(2,0),(1,1)\}),$ then $\mathcal{U}*\mathcal{U}\subset\mathcal{T}$, and $|\mathcal{U}|=5$.
For $i$ and $j$ at most $3$, $p_{i,j}(x,y)$ is uniquely determined by $x^*y$ and $p_{i,j}(x,y)=p_{j,i}(x,y)$. The determinant of the matrix $G=(g_{k,l}(\alpha_i))_{\substack{4\leq i\leq 8\\ (k,l)\in\mathcal{U}}}$ is $-16\mathrm{i}d/(d+1)^3$, so $X$ carries a nonsymmetric association scheme by Theorem~\ref{thm:S8}.
\end{example}

\section{Designs from association schemes}\label{sec:designsfromschemes}
So far we have focused on association schemes obtained from nice complex spherical designs.
In this section we consider the converse, namely sufficient conditions for obtaining complex spherical designs from association schemes. 

Every symmetric association scheme can be associated with a real spherical design in a natural way. Let $(X,\{R_i\}_{i=0}^s)$ be a symmetric scheme such that $E_1$ has rank $d:=m_1$.
Since a primitive idempotent $E_1$ is a positive semidefinite matrix,
there exists a $|X|\times d$ matrix $U$ such that $\frac{d}{|X|}E_1=UU^T$.
We identify elements of $X$ as rows of $U$.
If $E_1$ has no repeated columns, then this embedding the scheme into real unit sphere $S^{d-1}$ is injective.
It is known that $X$ is always real spherical $2$-design of degree at most $s$, and it is $3$-design if and only if $q_{1,1}^1=0$ \cite{cgs}.
In case of $Q$-polynomial association schemes, see \cite{suda}.

Here, we consider the case of a nonsymmetric association scheme, which we associate with a complex spherical design in a natural way.
Let $(X,\{R_i\}_{i=0}^s)$ be a nonsymmetric scheme such that $E_1$ has rank $d:=m_1$ and has no repeated rows. Again we identify elements of $X$ as rows of $U$, where $\frac{d}{|X|}E_1=UU^*$.
This embedding of $X$ into $\Om(d)$ is injective. 
We will see (Corollary~\ref{cor:Tscheme}) that for $E_1^T\neq E_1$, $X$ is always a $\cT$-design where $\cT$ contains $\{(i,j)\in\nn^2: i+j\leq2\}$. Moreover, whether or not $\cT$ contains $(2,1)$ or $(3,0)$ depends on whether or not $q_{1,1}^{1} = 0$ or $q_{1,1}^{\widehat{1}} = 0$ respectively.

In fact, when a $\cT$-design in $\Om(d)$ carries an association scheme, the following theorem shows that we can characterize the integer pairs $(i,j)$ in $\cT$ using the Krein parameters of the scheme. Recall that $\widehat{h}$ is the index such that $E_{\widehat{h}} = E_h^T$.

\begin{theorem}\label{designAS}
Let $(X,\{R_i\}_{i=0}^s)$ be an association scheme, and identify the points of $X$ with unit vectors in $\Om(d)$ whose Gram matrix is a scalar multiple of $E_1$. Then $X$ is a $\mathcal{T}$-design in $\Om(d)$ if and only if for each $(i,j)\in\mathcal{T}$, the following holds:
$$\sum_{l_0,\ldots,l_i,h_0,\ldots,h_j=0}^sq_{0,0}^{l_0}q_{1,l_0}^{l_1}\cdots q_{1,l_{i-1}}^{l_i}q_{0,0}^{h_0}q_{1,h_0}^{h_1}\cdots q_{1,h_{j-1}}^{h_j}q_{l_i,\widehat{h_j}}^0=\begin{cases}\frac{d^{2i}}{\binom{d+i-1}{i}} &\text{ if } i=j,\\0 &\text{ if }i\neq j.  \end{cases}$$
\end{theorem}

\begin{proof}
Let $(x,y)$ be in $R_n$ and set $i\geq 0$. By comparing the $(x,y)$-entry of both sides of
\[
(|X|E_1)^i=\sum_{l_0,l_1,\ldots,l_i=0}^sq_{0,0}^{l_0}q_{1,l_0}^{l_1}\cdots q_{1,l_{i-1}}^{l_i}|X|E_{l_i},
\]
we obtain
\[
(Q_{n1})^i=\sum_{l_0,l_1,\ldots,l_i=0}^sq_{1,0}^{l_0}q_{1,l_0}^{l_1}\cdots q_{1,l_{i-1}}^{l_i}Q_{nl_i}.
\]
Similarly, for $j \geq 0$,
\[
(\overline{Q_{n1}})^j=\sum_{h_0,h_1,\ldots,h_j=0}^sq_{0,0}^{h_0}q_{1,h_0}^{h_1}\cdots q_{1,h_{j-1}}^{h_j}\overline{Q_{nh_j}}.
\]
Combining these two equations,
\begin{align*}
\frac{1}{|X|^2}&\sum\limits_{x,y\in X}(x^*y)^i(\overline{x^*y})^j\\
 &=\frac{1}{|X|}\sum_{n=0}^s\left(\frac{Q_{n1}}{d}\right)^i\left(\frac{\overline{Q_{n1}}}{d}\right)^jk_n \\
&=\frac{1}{|X|d^{i+j}}\sum_{l_0,\ldots,l_i,h_0,\ldots,h_j=0}^sq_{0,0}^{l_0}q_{1,l_0}^{l_1}\cdots q_{1,l_{i-1}}^{l_i}q_{0,0}^{l_0}q_{1,h_0}^{l_1}\cdots q_{1,h_{j-1}}^{h_j} \left(\sum_{n=0}^s Q_{nl_i}\overline{Q_{nh_j}}k_n\right)\\
&=\frac{1}{d^{i+j}}\sum_{l_0,\ldots,l_i,h_0,\ldots,h_j=0}^sq_{0,0}^{l_0}q_{1,l_0}^{l_1}\cdots q_{1,l_{i-1}}^{l_i}q_{0,0}^{l_0}q_{1,h_0}^{l_1}\cdots q_{1,h_{j-1}}^{h_j}q_{l_i,\widehat{h_j}}^0.
\end{align*}
The result now follows from Lemma~\ref{lem:moment}.
\end{proof}

\begin{corollary}\label{cor:Tscheme}
Let $(X,\{R_i\}_{i=0}^s)$ be an association scheme and identify the points of $X$ with unit vectors whose Gram matrix is a scalar multiple of $E_1$. Then $X$ is a complex spherical $\mathcal{T}$-design such that:
\begin{enumerate}[(i)]
\item $\mathcal{T}$ contains $\cl\{(1,1)\}$;
\item $(2,0)\in\cT$ if and only if $\widehat{1}\neq 1$; 
\item $(2,1)\in \mathcal{T}$ if and only if $q_{1,1}^{1}=0$;
\item $(3,0)\in \mathcal{T}$ if and only if $q_{1,1}^{\widehat{1}}=0$.
\end{enumerate}
\end{corollary}
\begin{proof}
This follows immediately from Theorem~\ref{designAS}. 
\end{proof}

In certain cases, we can completely characterize the association scheme associated with a complex spherical design or vice versa. Theorems \ref{thm:schemedesign1} and \ref{thm:schemedesign2} are two examples.

\begin{theorem}\label{thm:schemedesign1}
Let $\mathcal{U}=\{(i,j)\in\nn^2: i+j\leq1\}$, and let $X \sbs \Om(d)$ be a tight design with respect to $\cU$ with degree $2$. Then the inner product relations in $X$ define a $2$-class nonsymmetric association scheme with second eigenmatrix 
$$
Q=\begin{pmatrix}1&d&d\\1&\frac{-1+\sqrt{-1-2d}}{2}&\frac{-1-\sqrt{-1-2d}}{2}\\1&\frac{-1-\sqrt{-1-2d}}{2}&\frac{-1+\sqrt{-1-2d}}{2}  \end{pmatrix}.
$$
Conversely, if $(X,\{R_i\}_{i=0}^2)$ is a $2$-class nonsymmetric association scheme with second eigenmatrix $Q$ above, then the primitive idempotent $E_1$ is a scalar multiple of the Gram matrix of a tight design with respect to $\cU$ with degree $2$.
\end{theorem}

\begin{proof}
Let $X$ be a tight design with respect to $\cU$ with degree $s=2$. By Theorem~\ref{thm:S1}, $X$ carries a scheme. 
Let $A(X)=\{a_1+\textrm{i}b_1,a_2+\textrm{i}b_2\}$, for some $a_1,b_1,a_2,b_2\in \mathbb{R}$. From Corollary~\ref{S12}, we get that $1 +2da_i = 0$, and so $a_1 = a_2 = -1/2d$. Since $A(X)$ is closed under complex conjugation, we see that $A(X)=\{a+\textrm{i}b,a-\textrm{i}b\}$ with $a =  -1/2d$ and $b>0$. Let $k$ be the valency of each of the two relations. From the fact that $X$ is a $\cU*\cU$-design we get the following equations (among others):
\begin{align*}
&1+2k =|X|;\\
&1+2ak=0;\\
&1+2a^2k-2b^2k=0.
\end{align*} 
These equations give the unique solution $k=d$, $b=\sqrt{1+2d}/2d$ and $|X| = 2d+1$.
By equation~\eqref{eqn:secondeigenmatrix}, the second eigenmatrix is 
$$Q=\begin{pmatrix}1&d&d\\1&\frac{-1+\sqrt{-1-2d}}{2}&\frac{-1-\sqrt{-1-2d}}{2}\\1&\frac{-1-\sqrt{-1-2d}}{2}&\frac{-1+\sqrt{-1-2d}}{2}  \end{pmatrix},$$
which implies that the scheme $X$ coincides with a $2$-class nonsymmetric association scheme.

Conversely, let $(X,\{R_i\}_{i=0}^2)$ be a $2$-class nonsymmetric association scheme having the second eigenmatrix as above. Then the primitive idempotent $E_1$ is positive semidefinite with rank $d$. Considering the rows of a $|X| \times d$ matrix $U$ such that $\tfrac{d}{|X|}E_1=UU^*$, we obtain a set $X$ with inner product set $A(X)=\{\frac{-1+\sqrt{-1-2d}}{2d},\frac{-1-\sqrt{-1-2d}}{2d}\}$.
Then it is easy to see that $X$ is a tight design with respect to $\cU$.
\end{proof}

\begin{theorem}\label{thm:schemedesign2}
Let $\mathcal{T}=\{(i,j)\in\nn^2: i+j\leq3\}$, and let $X = Y \cup (-Y)$ be a $2$-antipodal $\mathcal{T}$-design in $\Omega(d)$ with degree $3$. If $G$ is Gram matrix of $Y$, then $\mathrm{i}\sqrt{2d-1}(G-I)$ is a skew-symmetric conference matrix. Conversely, if $C$ is a $2d \times 2d$ skew symmetric conference matrix, then $I+\frac{\mathrm{i}}{\sqrt{2d-1}}C$ is the Gram matrix of a set of vectors $L \sbs \Omega(d)$ such that $L \cup (-L)$ is a $\mathcal{T}$-design with degree $3$.
\end{theorem}

\begin{proof}
Let $X$ be a $2$-antipodal $\mathcal{T}$-design in $\Omega(d)$ with degree $s=3$. 
If we take $\mathcal{U}=\{(i,j)\in\nn^2: i+j\leq1\}$, then $\mathcal{U}*\mathcal{U}\subset \mathcal{T}$ and $s=|\mathcal{U}|$.
Therefore $X$ carries an association scheme by Theorem~\ref{thm:S1}.

We calculate the second eigenmatrix. 
Let $A(X)=\{a_1+\textrm{i}b_1,a_2+\textrm{i}b_2,-1\}$ for some $a_1,b_1,a_2,b_2\in \mathbb{R}$.
Since $A(X)$ is closed under conjugation, there are two possibilities: 
\begin{enumerate}
\item $b_1\neq0$ and $a_1+\textrm{i}b_1=a_2-\textrm{i}b_2$,
\item $b_1=b_2=0$ and $a_1\neq a_2$.
\end{enumerate}
The second case does not occur, since $\sum_{x,y\in X}(x^*y)^2=\tfrac{2}{d(d+1)}\sum_{x,y\in X}g_{2,0}(x^*y)=0$. So $A(X)=\{a+\textrm{i}b,a-\textrm{i}b,-1\}$ with $b>0$.
Letting $k$ be the valency of the relations defined by $a+\textrm{i}b$ and $a-\textrm{i}b$, we get the following equations:
\begin{align*}
&2+2k =|X|;\\
&2ak=0;\\
&2+2a^2k-2b^2k=0;\\
&2d-2-2k+2(a^2+b^2)dk=0.
\end{align*} 
The unique solution to these equations is $a=0$, $b=1/\sqrt{2d-1}$ and $k=2d-1$, and thus the second eigenmatrix of the scheme is 
$$Q=\begin{pmatrix}1&d&d&2d-1\\1&\frac{\textrm{i}d}{\sqrt{2d-1}}&\frac{-\textrm{i}d}{\sqrt{2d-1}}&-1\\1&\frac{-\textrm{i}d}{\sqrt{2d-1}}&\frac{\textrm{i}d}{\sqrt{2d-1}}&-1 \\1&-d&-d&2d-1  \end{pmatrix}.$$
Let $A$ be the principal submatrix of $4E_1$ restricted to $Y$, where $X=Y\cup(-Y)$.
Then $4E_1=\begin{pmatrix}A&-A\\-A&A \end{pmatrix}$, 
and $E_1^2=E_1$ implies that $A^2=2A$.
Define a matrix $C$ such that  $A=I+\frac{\textrm{i}}{\sqrt{2d-1}}C$. Then $C$ is a $\pm 1$ matrix satisfying  $C^T=-C$ and $CC^T=(2d-1)I$, so $C$ is a skew symmetric conference matrix.

Conversely, let $C$ be a $2d\times 2d$ skew symmetric conference matrix. Then the complex matrix $I+\frac{\mathrm{i}}{\sqrt{2d-1}}C$ is Hermitian positive semi-definite of rank $d$.
Regard this matrix as the Gram matrix of unit vectors $L=\{x_1,\ldots,x_{2d}\}$ in $\Omega(d)$, and define $X=L\cup (-L)$. Then $A(X)=\{\pm\frac{\textrm{i}}{\sqrt{2d-1}},-1\}$, so $X$ has  degree $s=3$.
Set $$(\alpha_i)_{0\leq  i\leq3}=(1,\tfrac{\textrm{i}}{\sqrt{2d-1}},-\tfrac{\textrm{i}}{\sqrt{2d-1}},-1).$$
Then $X$ is inner product invariant with valencies $(k_i)_{0\leq i\leq 3}=(1,2d-1,2d-1,1)$, and
 $X$ is a $\mathcal{T}$-design in $\Omega(d)$ where $\mathcal{T}=\{(k,l)\in\nn^2: k+l\leq 3\}$.
\end{proof}

\section{Derived codes and designs}\label{sec:derived}
In this section we develop the notion of a derived complex spherical code, similar to that of a derived real spherical code \cite{dgs}.  Fix a point $z$ in a code $X \sbs \Om(d)$ and an angle $\al \in A(X)$ such that $\abs{\al} < 1$, and consider the points $R_\alpha(z):=\{y\in X \mid \langle z,y\rangle =\alpha\}$. The \defn{derived code} $X_{\alpha}(z)$ is the orthogonal projection of $R_\alpha(z)$ onto $z^\perp=\{y\in \mathbb{C}^d \mid \langle z,y\rangle =0 \}$, with points rescaled to lie in $\Omega(d-1)$.

After a unitary transformation, we may assume $z=(1,0,\ldots,0)$.
Then 
$$X_\alpha(z)=\{(x_2,\ldots,x_d)\in\Omega(d-1):  (\alpha,x_2\sqrt{1-|\alpha|^2},\ldots,x_d\sqrt{1-|\alpha|^2})\in X\},$$
and its inner product set $A(X_\alpha(z))$ is contained in
$$\Big\{\frac{\beta-|\alpha|^2}{1-|\alpha|^2}: \beta\in A(X)\Big\}.$$
If $X$ is an $\cS$-code in $\Om(d)$ with an annihilator polynomial $F(x)$,
then $X_{\al}(z)$ is an $\cS$-code in $\Om(d-1)$ with an annihilator polynomial $G(x):=F((1-\abs{\al}^2) x+\abs{\al}^2)$.

Define $A^*(X)=\{ \al \in A(X): \abs{\al}<1\}$ and denote the cardinality of $A^*(X)$ by $s^*$.

\begin{theorem}\label{derived}
Let $X$ be a $\mathcal{T}$-design in $\Omega(d)$ with degree $s$.
Let $\cS^*$ be a lower set such that $|\cS^*|=s^*$ and 
$\mathcal{T}'$ a lower set such that $\cT'*\cS^*\sbs \cT$. 
Assume the matrix $G=(\al^m\co{\al}^n)_{\substack{\al\in A^*(X)\\ (m,n)\in\cS^*}}$ is nonsingular.
Then for any $z\in X$ and any $\alpha\in A^*(X)$, $X_{\alpha}(z)$ is a $\mathcal{T}'$-design in $\Omega(d-1)$.
\end{theorem}

\begin{proof}
Without loss of generality assume $z=e_1\in X$.
Fix $(k,l)\in\mathcal{T}'$.
For any $F_{k,l}\in\text{Hom}_{d-1}(k,l)$ and any $(m,n)\in S^*$,
define $G_{k,l}^{m,n}\in\text{Hom}_d(k+m,l+n)$ as
$$G_{k,l}^{m,n}(x_1,\ldots,x_d)=x_1^m\bar{x_1}^nF_{k,l}(x_2,\ldots,x_d).$$ 
Then 
\begin{align*}
\sum_{x\in X}G_{k,l}^{m,n}(x)-\sum_{x\in X:|\langle x,z\rangle|=1}G_{k,l}^{m,n}(x)
&=\sum_{\alpha\in A^*(X)}\sum_{x\in R_\alpha(z)}\alpha^m\bar{\alpha}^nF_{k,l}(x) \\
&=\sum_{\alpha\in A^*(X)}\alpha^m\bar{\alpha}^n(1-|\alpha|^2)^{\frac{k+l}{2}}\sum_{x\in X_\alpha(z)}F_{k,l}(x)
\end{align*}
For any element $U$ in the unitary group $U(d)$ such that $Uz=z$, we consider the action of $U$.
Since $X$ is a $\mathcal{T}$-design, the LHS is invariant under the action of $U$.
Therefore the RHS 
$$\sum_{\alpha\in A^*(X)}\alpha^m\bar{\alpha}^n(1-|\alpha|^2)^{\frac{k+l}{2}}\sum_{x\in UX_\alpha(z)}F_{k,l}(x)$$
is independent of $U\in U(d)$.
The $s^*$ elements $(k,l)\in S^*$ yields a linear equation whose unknowns are $\sum_{x\in UX_\alpha(z)}F_{k,l}(x)$ for $\alpha\in A^*(X)$.
Its coefficient matrix 
$$\text{Diag}((1-|\alpha|^2)^{\frac{k+l}{2}})_{\alpha \in A^*(X)}\cdot G$$
is nonsingular. 
Thus $\sum_{x\in UX_\alpha(z)}F_{k,l}(x)$ does not depend on $U\in U(d)$, which implies $X_\alpha(z)$ is a $\mathcal{T}'$-design in $\Omega(d-1)$. 
\end{proof}

\begin{example}({\bf Derived codes of SIC-POVMs in $\cx^2, \cx^8$})
For $X$ in Example~\ref{ex:sicpovm} and any $\al\in A^*(X)$ and any $z\in X$, we consider the derived code $Y=X_{\alpha}(z)$ in $\Omega(d-1)$.
Then $|Y|=d^2-1$ and $Y$ is a $\cT=\cl(\{(3,3)\})$-design with degree $s=4$.
Taking $\cU=\cl(1,1)\in \cT$ so that $\cU*\cU\sbs \cT$ and $|\cU|=4$,
then by Theorem~\ref{thm:S1}, $X$ carries a nonsymmetric association scheme.
\end{example}

\begin{example}\label{80}
For $X$ in Example~\ref{756} and any $7\leq i\leq 12$ and any $z\in X$, we consider the derived code $Y=X_{\alpha_{i}}(z)$ in $\Omega(5)$.
Then $|Y| = 80$, $A(Y)=\{\pm\frac{1}{3},\pm\frac{\mathrm{i}}{\sqrt{3}},-1\}$ with degree $s=5$, and $Y$ is $\mathcal{T}$-design where $\cT=\cl(\{(3,2),(2,3)\})$.
Taking $\cU = \cl(\{(1,1),(0,2)\})$
so that $\cU*\cU\sbs \cT$ and $|\mathcal{U}|=s$,
then by Theorem~\ref{thm:S1}, $X$ carries a nonsymmetric association scheme.
\end{example}

\begin{example}\label{270}
For $X$ in Example~\ref{756} and any $z\in X$, we consider the derived code $Y=X_{\alpha_6}(z)$ in $\Omega(5)$.
Then $|Y| = 270$, $A(Y)=\{0,w_6^j,\frac{1}{2}w_6^j: 0\leq j\leq5\}\setminus\{1\}$ with degree $s=12$, and $X$ is $\mathcal{T}$-design where $\cT=\cl(\{(5,2),(2,5)\})$.
Set
\[
(\al_i)_{i=0}^{12} = (\tfrac{1}{2},\tfrac{1}{2}\om,\ldots,\tfrac{1}{2}\om^5,0,1,\om,\ldots,\om^5).
\]
We take 
$\cU=\cl(\{(3,0),(1,1),(0,2)\})$,
so the determinant of 
$G=(g_{k,l}(\alpha_i))_{\substack{1\leq i\leq 7\\ (k,l)\in\mathcal{U}}}$ is $-27/256$ and $X$ carries a nonsymmetric association scheme by Theorem~\ref{thm:S8}.
\end{example}

\begin{example}
For any $z\in X$ and $\al\in A^*(X)$ in Example~\ref{thm:schemedesign1}, 
consider the derived codes $X_\al(z)$.
Take $\cS=\{(0,0),(1,0)\}$, 
then the determinant of the matrix $G=(\al^m\co{\al}^n)_{(m,n)\in\cS,\al\in A^*(X)}$ is $-2\mathrm{i}/\sqrt{2d-1}$.
Theorem~\ref{derived} shows that $X_\al(z)$ is $\cT'$-design, where $\cT'=\{(k,l)\in\nn^2:k+l\leq 2\}$.
The derived design $X_\al(z)$ coincides with the design appearing in Example~\ref{thm:schemedesign2}.      
As is shown, both designs carry schemes.
The scheme obtained from $X_\al(z)$ is a subconstituent of the scheme obtained from $X$. 
\end{example}

\begin{example}({\bf Derived designs of complex MUBs in $\cx^{2^{m}}$, $m$ even})
For any $z\in X$ and $\al\in A^*(X)$ with $|\al|=1/\sqrt{d}$ in Example~\ref{complexMUBodd}, 
consider the derived codes $X_\al(z)$.
Take $\cS=\cl(\{(1,1),(2,0)\})$, 
then the determinant of the matrix $G=(\al^m\co{\al}^n)_{(m,n)\in\cS,\al\in A^*(X)}$ is $-16/d^3$.
Theorem~\ref{derived} shows that $X_\al(z)$ is $\cT'$-design, where $\cT'=\{(k,l)\in\nn^2:k+l\leq 3\}$.
\end{example}

\begin{example}({\bf Derived designs of complex MUBs in $\cx^{2^{m}}$, $m$ odd})
For any $z\in X$ and $\al\in A^*(X)$ with $|\al|=1/\sqrt{d}$ in Example~\ref{complexMUBeven}, 
consider the derived codes $X_\al(z)$.
Take $\cS=\cl(\{(1,1),(2,0)\})$, 
then the determinant of the matrix $G=(\al^m\co{\al}^n)_{(m,n)\in\cS,\al\in A^*(X)}$ is $16\mathrm{i}/d^3$.
Theorem~\ref{derived} shows that $X_\al(z)$ is $\cT_1$-design, where $\cT_1=\cl(\{(3,2)\})$.
Similarly taking $\cS=\cl(\{(1,1),(2,0)\})$ yields that $X_\al(z)$ is $\cT_2$-design, where $\cT_2=\cl(\{(2,3)\})$.
Thus $X_\al(z)$ is $\cT$-design where $\cT_2=\cl(\{(3,2),(2,3)\})$.
\end{example}

\section{Designs from subgroups of the unitary group}\label{sec:designfromorbit}

Let $G$ be a finite subgroup of the unitary group $U(d)$ acting on $\Om(d)$. In this section, we consider the conditions under which an orbit of $G$, say $Gx := \{gx : g \in G\}$ for some $x \in \Om(d)$, is a spherical design. 

Recall that $\Hom(k,l)$ is a representation of $U(d)$ with the following action: for $g \in U(d)$ and $f \in \Hom(k,l)$, 
\[
(gf)(z) = f(g^{-1}z).
\]
Given a vector space $V$ which is a representation of $G$, the \defn{stabilizer} of $G$ is
\[
V^G := \{f \in V: gf = g \text{ for all } g \in G\}.
\]
Clearly $V^G$ is a subspace of $V$. The dimension of $V^G$ is the number of times the trivial representation of $G$ appears in $V$. For more background on representation theory, see \cite{fh}.

\begin{theorem}\label{thm:harmdesign}
Let $\cU$ be a lower set in $\nn^2$ and $G$ a finite subgroup of $U(d)$. 
The following are equivalent:
\begin{enumerate}[(i)]
\item For all $x\in \Om(d)$, $Gx$ is a complex spherical $\cU$-design.
\item $\Harm(k,l)^G=\{0\}$ for all $(k,l) \neq (0,0)$ in $\cU$.
\end{enumerate}
\end{theorem}

\begin{proof}
(i)$\Rightarrow$(ii):
Assume (i).
For $(k,l) \neq (0,0)$ in $\cU$, $f\in \Harm(k,l)^G$ and $x\in \Om(d)$,
\begin{align*}
f(x)=\frac{1}{|G|}\sum_{g\in G}(gf)(x)=\frac{1}{|G|}\sum_{g\in G}f(g^{-1}x)=\frac{1}{|Gx|}\sum_{y\in Gx}f(y)=0.
\end{align*}
Hence $f=0$, which implies (ii). 

(ii)$\Rightarrow$(i):
Let $f$ be a polynomial in $\Harm(k,l)$. Then
\[
\frac{1}{|Gx|}\sum_{y\in Gx}f(y)=\frac{1}{|G|}\sum_{g\in G}f(g^{-1}x)= \frac{1}{|G|}\sum_{g\in G}(gf)(x)= 0,
\]
since $\frac{1}{|G|}\sum_{g\in G}(gf)$ is in $\Harm(k,l)^G$. Hence $Gx$ is a complex spherical $\cU$-design.
\end{proof}

If, for each $(k,l) \in \cU$, the representation of $G$ on $\Harm(k,l)$ is irreducible, then $\Harm(k,l)^G = \{0\}$ and so $Gx$ is a $\cU$-design. But, in fact more is true.

\begin{theorem}\label{thm:irrdesign}
Let $G$ be a finite subgroup of $U(d)$, let $\cU$ be a lower set in $\nn^2$.
Assume for each $(k,l)\in \cU$, the representation of $G$ on $\Harm(k,l)$ is irreducible.
Then for each vector $x\in\Om(d)$, the orbit $Gx$ is a complex spherical $\cU*\cU$-design. 
\end{theorem}


\begin{proof}
Recall that for functions $f$ and $h$ on $\Om(d)$, we have an inner product $\ip{f}{h} =  \int_{\Om(d)} \overline{f(z)}h(z) \dd z$, and given a subset $X = Gx$ we may define a second inner product $\ip{f}{h}_X =  \tfrac{1}{|X|} \sum_{z \in X} \overline{f(z)}h(z)$. To show that $X$ is a $\cU*\cU$-design, we must show that $\ip{1}{f} = \ip{1}{f}_X$ for all $f \in \bigoplus_{(k,l) \in \cU*\cU} \Hom(k,l)$. Since $\oplus_{(k,l)\in\cU*\cU}\Hom(k,l)$ is generated by the products $\Hom(k,l)\Hom(l',k')$ as $(k,l)$ and $(k',l')$ run through $\cU$, it follows that $X$ is a $\cU*\cU$-design if and only if 
\[
\ip{f}{h} = \ip{f}{h}_X
\]
for all $f,h \in \bigoplus_{(k,l) \in \cU} \Hom(k,l)$.

The first inner product is $U(d)$-invariant (and therefore $G$-invariant), while the second is only $G$-invariant. If $(k,l) \neq (k',l')$, then $\Harm(k,l)$ and $\Harm(k',l')$ are distinct irreducible representations of both $G$ and $U(d)$. Therefore, for $f \in \Harm(k,l)$ and $h \in \Harm(k',l')$ we have $\ip{f}{h} = 0 = \ip{f}{h}_X$. So, it suffices to consider $f$ and $h$ from the same $\Harm(k,l)$. In this case, since $\Harm(k,l)$ is an irreducible representation of $G$, we know that by Schur's Lemma there is a unique (up to a scalar multiple) $G$-invariant inner product on $\Harm(k,l)$. (See for example Sepanski Cor. 2.20.) Therefore $\ip{f}{h} = c\ip{f}{h}_X$ for some constant $c$. Since $\ip{f}{h}$ is $U(d)$-invariant, the constant is independent of the choice of $x$, and normalization implies that $c = 1$. Thus $\ip{f}{h} = \ip{f}{h}_X$. 
\end{proof}

\begin{example}
Let $G$ be the Pauli group of $U(d)$, that is, $G$ is generated by 
$P_x=(\de_{i+1,j})_{1\leq i,j\leq d}$ (indices run through $\{k\mod d:1\leq k\leq d\}$) and 
$P_z=\mathrm{diag}(1,w,\ldots,w^{d-1})$ where $w$ is a primitive $d$-th root of unity.
For $(k,l)\in\{(0,0),(1,0),(0,1)\}$, $\Harm(k,l)$ is an irreducible representation of $G$.
Therefore for any vector $x\in\Om(d)$, $Gx$ is a complex spherical $\cT$-design, where $\cT=\{(k,l)\in\nn^2:k+l\leq2\}$.
\end{example}

In order to check that $\Harm(k,l)^G = \{0\}$ in Theorem~\ref{thm:harmdesign} or $\Harm(k,l)$ is irreducible in Theorem~\ref{thm:irrdesign}, we use characters and dimension arguments. The following lemma is standard representation theory \cite{fh,sep}.

\begin{lemma}\label{lem:charrep}
Let $V$ be a representation space of $G$ with character $\chi: G \rightarrow \cx$. Then
\begin{itemize}
\item[(i)] $\dim(V^G) = \frac{1}{|G|} \sum_{g \in G} \chi(g)$, and
\item[(ii)] $V$ is irreducible if and only if $\frac{1}{|G|} \sum_{g \in G} \abs{\chi(g)}^2 = 1$.
\end{itemize}
\end{lemma}

To find the dimension of $\Harm(k,l)^G$, it suffices to find the dimension of $\Hom(k,l)^G$.

\begin{lemma}\label{harminvariant}
\[
\dim_{\cx}(\Harm(k,l)^G)=\dim_{\cx}(\Hom(k,l)^G)-\dim_{\cx}(\Hom(k-1,l-1)^G).
\]
\end{lemma}

\begin{proof}
The Laplacian operator $\De = \sum_{i=1}^d \partial^2/\partial z_i \partial \overline{z_i}$ maps $\Hom(k,l)$ onto $\Hom(k-1,l-1)$ with kernel $\Harm(k,l)$. Moreover, it is not difficult to check that $\De$ commutes with the action of each $U \in U(d)$ on functions on $\Om(d)$. Therefore $\De: \Hom(k,l)^G \rightarrow \Hom(k-1,l-1)^G$ is surjection and has kernel $\Harm(k,l)^G$, so the result follows.
\end{proof}

In order to find the dimension of $\Hom(k,l)^G$, we use a Molien-series theorem that gives a generating function for the dimension for all $k$ and $l$.

\begin{theorem}\label{thm:molien}
Let $G$ be a finite subgroup of $U(d)$. Then
\[
\sum_{k,l=0}^\infty \dim_{\cx}(\Hom(k,l)^G)x^ky^l =\frac{1}{\abs{G}}\sum_{g\in G}\frac{1}{\det(I-xg)\det(I-y\co{g})}.
\]
\end{theorem}

\begin{proof}
For convenience, set $S_{k,l}=\Hom(k,l)$ with representation $\rho_{k,l}:G \rightarrow S_{k,l}$ and character $\chi_{k,l}$. We will write $\bla^{\bal}=\la_1^{\al_1}\cdots\la_d^{\al_d}$, $\co{\bla}^{\bbe}={\co{\la_1}}^{\be_1}\cdots{\co{\la_d}}^{\be_d}$ and $\abs{\bal}=\sum_{i=1}^d\al_i$.

Fix $g\in G$, and let $\la_1,\ldots,\la_d$ be the eigenvalues of $g = \rho_{1,0}(g)$, with eigenvectors $v_1,\ldots,v_d \in S_{1,0}$. Since $S_{k,0} = \Sym^k(S_{1,0})$, the space of $k$-th symmetric powers of $S_{1,0}$, it follows that the $k$-th symmetric powers of $\{v_1,\ldots,v_d\}$ form a basis for $S_{k,0}$, and each basis vector is an eigenvector for $\rho_{k,0}(g)$. Therefore $\rho_{k,0}(g)$ has eigenvalues $\bla^{\bal}$, where $\bal = (\al_1,\ldots,\al_d) \in \nn^d$ with $\abs{\al} = k$. Similarly, $\rho_{0,1}(g)$ has eigenvalues $\co{\la_1},\ldots,\co{\la_d}$, and $\rho_{0,l}(g)$ has eigenvalues ${\co{\bla}}^{\bbe}$, for $\abs{\bbe} = l$. Therefore $\rho_{k,l}(g)$ has eigenvalues $\bla^{\bal}{\co{\bla}}^{\bbe}$, and
\begin{equation}\label{eqn:symchar}
\chi_{k,l}(g)=\sum_{\abs{\bal}=k,\abs{\bbe}=l}\bla^{\bal}{\co{\bla}}^{\bbe}
\end{equation}

Next, observe that 
\begin{align*}
\sum_{\bal \in{\nn}^d}\bla^{\bal}x^{\abs{\bal}}&=\sum_{\bal \in{\nn}^d}(\la_1 x)^{\al_1}\cdots (\la_d x)^{\al_d}\\
&=\prod_{i=1}^d \sum_{\al_i=0}^{\infty}(\la_i x)^{\al_i}\\
&=\prod_{i=1}^d \frac{1}{1-\la_i x}\\
&=\frac{1}{\det(I-xg)},
\end{align*}
from which we get 
\begin{equation}\label{eqn:ggenfunc}
\sum_{\bal,\bbe\in{\nn}^d}\bla^{\bal}{\co{\bla}}^{\bbe}x^{\abs{\bal}}y^{\abs{\bbe}}=\frac{1}{\det(I-xg)\det(I-y\co{g})}.
\end{equation}
We are now ready to prove Molien's theorem:
\begin{align*}
\sum_{k,l=0}^\infty \dim_{\cx}(S_{k,l}^G)x^ky^l&=\frac{1}{\abs{G}}\sum_{k,l=0}^\infty \sum_{g\in G}\chi_{k,l}(g)x^ky^l\quad (\text{by Lemma \ref{lem:charrep})}\\
&=\frac{1}{\abs{G}}\sum_{k,l=0}^\infty \sum_{g\in G}\sum_{\abs{\bal}=k,\abs{\bbe}=l}\bla^{\bal}{\co{\bla}}^{\bbe} x^ky^l\quad (\text{by equation \ref{eqn:symchar}})\\
&=\frac{1}{\abs{G}}\sum_{g\in G}\frac{1}{\det(I-xg)\det(I-y\co{g})}\quad (\text{by equation \ref{eqn:ggenfunc}}).
\end{align*} 
\end{proof}

Therefore we also obtain a generating function for the harmonic polynomials, which gives an easy method of checking if $\Harm(k,l)^G = \{0\}$ (as in Theorem~\ref{thm:harmdesign}):

\begin{corollary}
Let $G$ be a finite subgroup of $U(d)$. Then $\dim_{\cx}(\Harm(k,l)^G)$ is the coefficient of $x^ky^l$ in 
\begin{align*}
\frac{1}{\abs{G}}\sum_{g\in G}\frac{1-xy}{\det(I-xg)\det(I-y\co{g})}.
\end{align*} 
\end{corollary}

\begin{proof}
By Lemma~\ref{harminvariant} and Theorem~\ref{thm:molien},
\begin{align*}
\sum_{k,l=0}^\infty \dim_{\cx}(\Harm(k,l)^G)x^ky^l&=\sum_{k,l=0}^\infty (\dim_{\cx}(S_{k,l}^G)-\dim_{\cx}(S_{k-1,l-1}^G))x^ky^l\\
&=\sum_{k,l=0}^\infty \dim_{\cx}(S_{k,l}^G)x^ky^l-xy\sum_{k,l=0}^\infty \dim_{\cx}(S_{k,l}^G)x^ky^l\\
&=\frac{1}{\abs{G}}\sum_{g\in G}\frac{1-xy}{\det(I-xg)\det(I-y\co{g})}.
\end{align*}
\end{proof}

Likewise, there is a quick way of checking if $\Harm(k,l)$ is irreducible for all $(k,l)$ in a lower set $\cU$ (as in Theorem~\ref{thm:irrdesign}):

\begin{corollary}
Let $G$ be a finite subgroup of $U(d)$. Then $\Harm(k-i,l-i)$ is irreducible for all $i = 0,\ldots,\min(k,l)$ if and only if the coefficient of $x^kz^ky^lw^l$ in 
\begin{align*}
\frac{1}{\abs{G}}\sum_{g\in G}\frac{1}{\det(I-xg)\det(I-y\co{g})\det(I-z\co{g})\det(I-wg)}
\end{align*} 
is equal to $\min(k,l)+1$.
\end{corollary}

\begin{proof}
Without loss of generality assume $k \geq l$, and let $\chi_{k,l}$ be the character of $\Hom(k,l)$. Since $\chi_{k,l}(g)$ is the coefficient of $x^ky^l$ in $\frac{1}{\det(I-xg)\det(I-y\co{g})}$, (as in the proof above), it follows that $(\chi_{k,l},\chi_{k,l}) = \frac{1}{|G|}\sum_{g \in G} \abs{\chi_{k,l}(g)}^2$ is the coefficient of $x^kz^ky^lw^l$
\[
\frac{1}{\abs{G}}\sum_{g\in G}\frac{1}{\det(I-xg)\det(I-y\co{g})\det(I-z\co{g})\det(I-wg)}.
\]
But $(\chi_{k,l},\chi_{k,l}) = \sum_\mu m_\mu^2$, where $m_\mu$ is the number of times an irreducible representation $V_\mu$ occurs in $\Hom(k,l)$. At least $l+1$ irreducible representations occur in $\Hom(k,l)$, since $\Hom(k,l)$ is a direct sum of the $l+1$ representations $\Harm(k,l),\Harm(k-1,l-1),\ldots,\Harm(k-l,0)$. If some irrep occurred more than once, or more than $l+1$ irreps occurred, then we would have $\sum_\mu m_\mu^2 > l+1$. So each $\Harm(k-i,l-i)$ must be distinct and irreducible.
\end{proof}

\section{Designs from other combinatorial objects}\label{sec:other}
In this section we show some examples of complex spherical designs and nonsymmetric association schemes that do not arise from any of our previous results. In particular, we construct designs from Singer difference sets, mutually unbiased bases in odd dimensions, and orthogonal arrays. In the case of Singer difference sets and MUBs, the designs carry an association scheme.

\begin{example}{\bf(Singer Difference Sets)} 
Let $q$ be a prime power and consider the group $G = GF(q^3)^*/GF(q)^*$. If $\tr: GF(q^3) \rightarrow GF(q)$ denotes the finite field trace $\tr(x) = x + x^q + x^{q^2}$, then 
\[
D := \{xGF(q^*) : \tr(x) = 0\}
\]
is called the \defn{Singer difference set} of $G$ \cite{BJL}. Since $G \cong \zz_n$, $n = q^2+q+1$, we will write the group multiplication as addition. Since $\tr(x^q)= \tr(x)$ it follows that $qD = D$, so $q$ is a difference set multiplier for $D$. For $g \in G$, let $\chi_g: G \rightarrow \cx^*$ denote the character of $G$ indexed by $g$ (so $\chi_g(x) = \om^{gx}$, where $\om$ is an $n$-th primitive root of unity). Now consider $\chi_g|_D$, the restriction of $\chi_g$ to $D$, treated as a vector in $\cx^d$, $d = |D|$. We have 
\[
(\chi_g|_D)^*\chi_h|_D =  \sum_{d \in D} \om^{(g-h)d} = \sum_{d \in D} \chi_{g-h}(d),
\]
which we denote $\chi_{g-h}(D)$. It follows that the set 
\[
X = \{\frac{\chi_g|_D}{\sqrt{d}}: g \in G\}
\]
of normalized characters restricted to $D$ has inner product set 
\[
A(X) = \{\chi_g(D)/d: g \in G\}.
\]
Since $q$ is a multiplier of $D$, it follows that $\chi_g(D) = \chi_{qg}(D)$. We claim that $\chi_g(D) = \chi_h(D)$ if and only if $gD$ and $hD$ are the same orbit of the translation $x \mapsto qx$. For, if $\chi_g(D) = \chi_h(D)$, then $\om$ is a root of the polynomial
\[
p(x) = \sum_{d \in D} x^{gd} - \sum_{d\in D} x^{hd},
\]
which has degree at most $n-1$. But $p(x)$ is also a multiple of the minimal polynomial of $\om$, namely $\sum_{i=0}^{n-1} x^i$, which also has degree $n-1$. So the coefficients in $p(x)$ must be constant and the only constant possible is $0$. Thus $gD = hD$.

Finally, we note that $X$ carries an association scheme. This follows from the fact that the classes carried by $X$ are the orbitals of the group generated the actions $x \mapsto qx$ and $\{x \mapsto x+g: g \in G\}$ on $X$. (The scheme is commutative because it is a translation scheme~\cite[Section 2.10]{bcn}.)

To get a design, take an $n$-antipodal cover of $X$. From Lemma \ref{lem:projectivedesign}, such a cover is a $\cT$-design with $\cT = \{(k,l): k+l \leq 2\}$.

\end{example}

\begin{example}{\bf(MUBs in $\cx^p$, $p$ odd)} 
Let $p$ be an odd prime and let $\om$ be a $p$-th primitive root of unity. (In fact all the results in this example can be extended to prime power dimensions, but we show only the prime case here.) For $(i,y) \in \zz_p^2$, define vector $v_{i,y}$ in $\cx^p$:
\[
(v_{i,y})_x := \frac{1}{\sqrt{p}}\om^{ix^2 +yx} \quad (x \in \zz_p),
\]
and let $X = \{v_{i,y}: i,y \in \zz_p\}$.
Gauss sums can be used to evaluate the inner products of $X$ (see \cite[Theorem 5.15]{ln}):
\[
v_{i,y}^*v_{j,z} = \begin{cases}
c_p w^{-(z-y)^2/4(j-i)} (\frac{j-i}{p}), & i \neq j; \\
0, & i = j, y \neq z; \\
1, & i = j, y = z. 
\end{cases}
\]
Here $c_p$ is some constant with absolute value $\sqrt{p}$, and $(\frac{j-i}{p})$ is the Legendre symbol. It follows that the bases $B_i = \{v_{i,y}: y \in \zz_p\}$ are mutually unbiased. Combined with the standard basis, we get a complete set of MUBs \cite{gr,kr}. We claim that the inner product relations on $X$ are the orbitals of a transitive permutation group and therefore define an association scheme. 

First, notice that $v_{i,y}^*v_{j,z}$ depends only on $(j-i,z-y)$. Therefore the inner products are preserved by the group translations $(j,z) \mapsto (j+k,z+x)$. So consider $v_{0,0}^*v_{j,z}$. Note that $v_{0,0}^*v_{j,z} = v_{0,0}^*v_{i,y}$ if and only if $i$ and $j$ have the same quadratic character and $\tfrac{z^2}{4j} = \tfrac{y^2}{4i}$ (for $i,j \neq 0$). Thus the inner products are preserved by the permutations $(j,z) \mapsto (x^2j,xz)$, for $x \neq 0$ in $\zz_p$, and conversely if $\tfrac{z^2}{4j} = \tfrac{y^2}{4i}$ then some $x \neq 0$ satisfies $(x^2j,xz) = (i,y)$. Thus the relations defined by the inner products of $X$ are the orbitals of the group generated by 
\[
\{(j,z) \mapsto (j+k,z+x): k,x \in \zz_p\} \cup \{(j,z) \mapsto (x^2j,xz): x \in \zz_p, x \neq 0\}.
\]

To get a design, take an $p$-antipodal cover of $X$.  From Lemma \ref{lem:projectivedesign}, such a cover is a $\cT$-design with $\cT = \{(k,l): k+l \leq 2\}$.
\end{example}

Let $\mathcal{Q}=\{1,\ldots,q\}$, and $\mathcal{P}$ be the set of words of length $d$ over $\mathcal{Q}\cup\{0\}$.
For $u=(u_i)_{i=1}^d, v=(v_i)_{i=1}^d$, set $u\preceq v$ if and only if either
$u_i=0$ or $u_i=v_i$ for all $i$.
Then $(\mathcal{P},\preceq)$ forms a \defn{regular semilattice} with rank function $\mathrm{rank}(u)=|\{i:u_i\neq 0\}|$ (see \cite{del2}).
The set $X_r$ denotes the set of elements of rank $r$, where $0 \leq r\leq d$. 
\begin{definition}
Let $t,r$ be positive integers such that $t\leq r\leq d$.
A subset $X$ in  $X_r$ is said to be a \defn{nonbinary block $t$-design} if the number $\lambda_t=|\{x\in X:y\preceq x\}|$ does not depend on the choice of $y\in X_t$.
\end{definition}
When $r=d$, a nonbinary block $t$-design coincides with an orthogonal array with strength $t$. (Recall that an $n\times d$ array $X$ with entries from $\{1,\cdots,q\}$ is an \defn{orthogonal array}  $OA(n,d,q,t)$ if every $n\times t$ subarray of $X$ contains each $t$-tuple in $\{1,\cdots,q\}^t$ exactly $\lambda_t =n/q^t$ times as a row~\cite{BJL}.)
Define $\psi:\mathcal{P}\rightarrow \Om(d)$ as follows:
for a element $x$ in $\mathcal{P}$ with rank $r$, 
\begin{equation*}
\psi(x)_k=
\begin{cases}
\frac{1}{\sqrt{r}}w^{x_k} & \mbox{if }x_k \in\mathcal{Q}; \\
0& \mbox{if }x_k=0,
\end{cases}
\end{equation*}
where $w$ is a primitive $q$-th root of unity.
The image of $\psi$ for a subset $X$ in $\mathcal{P}$ is denoted by $\tilde{X}$. 
The following theorem constructs a small strength design from a nonbinary block designs.
\begin{theorem}
Let $X$ be a nonbinary block $t$-design $X$ in $X_r$ with $q\geq 3$ and $2\leq  t \leq  r \leq d$.
Then $\tilde{X}$ is a $\cT$-design in $\Om(d)$, where $\cT=\{(k,l)\in\nn^2:k+l\leq 2\}$.
\end{theorem}

\begin{proof}
For any $(k,l) \neq (0,0)$ in $\cT$, consider the average of a monomial $f(x)=\prod_{i=1}^dx_i^{a_i}\co{x_i}^{b_i}\in\Hom(k,l)$, where $\sum_{i=1}^d a_i=k$ and $\sum_{i=1}^d b_i=l$. If there is some $i$ such that $a_i \neq b_i$, then since a set of $\tilde{X}$ restricted to any at most two indices is $q$-antipodal it is easy to check that
\[
\frac{1}{|X|}\sum_{x\in \tilde{X}}f(x) = 0 =\int_{\Om(d)} f(z) \dd z.
\]
On the other hand, if $a_i = b_i$ for all $i$, then $(k,l) = (1,1)$ and $f(x) = x_j\co{x_j}$ for some $j$. In this case,
\[
\frac{1}{|X|}\sum_{x\in \tilde{X}}f(x) = \frac{q\la_1}{r|X|}, \quad \int_{\Om(d)} f(z) \dd z = \frac{1}{d}.
\]
Counting the size of the set $\{(y,x)\in X_t\times X:y\preceq x \}$ in two ways yields $|X|\binom{r}{t}=\binom{d}{t}q^t\lambda_t$, and in particular $|X|r = dq\la_1$.
Thus the two averages are equal, and $\tilde{X}$ is a $\cT$-design in $\Om(d)$. 
\end{proof}
Note that in the above theorem, $\cT$ does not contain $(2,2)$ in general.
Indeed, for $f(x)=(x_1\co{x_1})^2$, 
\begin{align*}
\frac{1}{|X|}\sum_{x\in \tilde{X}}f(x)=\frac{1}{r d}\neq\frac{2}{d(d+1)}\int_{\Om(d)} f(z) \dd z,
\end{align*}
unless $r = 2/(d+1)$.

\section{Acknowledgements}
The authors would like to thank Chris Godsil for his helpful discussions and Bill Martin for his invitation to the 2010 WPI Workshop on Schemes and Spheres, where this work originated. Aidan Roy was funded by NSERC, Industry Canada, Quantum Works, and the Fields Institute.
Sho Suda was supported by a JSPS research fellowship.
\bibliographystyle{plain}
\bibliography{nonsymmetricpaper}

\end{document}